\documentclass[]{amsart}
\usepackage{graphicx}
\usepackage[T1]{fontenc}
\usepackage[latin1]{inputenc}
\usepackage{enumerate}
\usepackage{pstricks}
\usepackage{latexsym,amssymb,amsthm,amsxtra,stmaryrd,mathptmx,accents}
\newcommand\car{{\mathbf 1}}

\newcommand\N{{\mathbb N}}

\newcommand\M{{\mathcal M}}

\renewcommand\P{{\mathbf P}} 

\newcommand{\A}{{\mathcal  A}}

\newcommand{\R}{{\mathbb R}}
\newcommand{\F}{{\mathcal F}}

\newcommand\esp[1]{{\mathbf E}\left[#1\right]}
\newcommand\pr[1]{{\mathbf P}\left[#1\right]}

\newcommand\prp[1]{{\mathbf P}\left[#1\right]}

\newcommand\hp{\hat{\mathbf P}}

\newcommand\bp{\bar{\mathbf P}}

\newcommand\pp{{\mathbf P}}
\newcommand\qp{{\tilde{\mathbf P}}}

\newcommand\qrp[1]{{\tilde{\mathbf P}\left[#1\right]}}

\newcommand{\pae}[1]{\mbox{$\lfloor \kern-1pt #1 \kern-1pt \rfloor$}}
\newcommand{\paep}[1]{\mbox{$\lceil \kern-1pt #1 \kern-1pt \rceil$}}

\newcommand\maG{{\mathcal G}}

\newcommand\maF{{\mathcal F}}
\newcommand\maB{{\mathcal B}}
\newcommand\maA{{\mathcal A}}
\newcommand\maI{{\mathcal I}}

\newcommand\om{{\omega}}

\def\N{{\mathbb N}}
\def\Z{{\mathbb Z}}

\def\P{{\mathbf P}}

\newcommand\tend{{\underset{n \rightarrow \infty}{\longrightarrow}}}

\newcommand\suite[1]{\left\{#1_n\right\}_{n \in \N}}
\newcommand\suiten[1]{\left\{#1\right\}_{n \in \N}}
\newcommand\suitez[1]{\left\{#1_n\right\}_{n \in \Z}}

{\bf}{\it}
\newtheorem{theorem}{Theorem}

\newtheorem{lemma}{Lemma}

{}
\newtheorem{proposition}{Proposition}
\newtheorem{cor}{Corollary}
\begin{document}
%\begin{frontmatter}
\title{A generalized backwards scheme for solving non monotonic stochastic recursions}

\author{P. Moyal}
\address{Laboratoire de Math\'ematiques Appliqu\'ees de Compi\`egne\\
Universit\'e de Technologie de Compiègne\\
D\'epartement G\'enie Informatique\\
Centre de Recherches de Royallieu\\
BP 20 529\\
60 205 COMPIEGNE Cedex\\
FRANCE\\
\emph{e-mail:} pascal.moyal@utc.fr}

%\date{Received: date / Revised: date}

\begin{abstract}
 We propose an explicit construction of a stationary solution for a stochastic recursion of the form 
 $X\circ\theta=\varphi(X)$ on a partially-ordered Polish space, when the monotonicity of $\varphi$ is not assumed. 
Under certain conditions, we show that an extension of the original probability space exists, on which a solution is well-defined, and 
construct explicitly this extension. We then provide conditions for the solution to be defined as well 
on the original space. We finally apply these results to the stability study of two non-monotonic queueing systems. 
%This result have natural applications in the stability study of discrete events dynamical systems among which, non conservative queueing systems 
%and 
% We then relate these results to existing ones, among which that of Anantharam \emph{et al.}, investigate the uniqueness of this solution, and focus on the special case of Markov chains. We finally provide applications of this result to queueing systems and random walks on graphs.  
\end{abstract}
%\begin{keyword}   
%\end{keyword} 
%\end{frontmatter}
%\subjclass[2000]{Primary : 60F17, Secondary :  60K25 \and 60B12}
\maketitle
\textbf{Keywords:} Stochastic recursions, Stationary solutions, Enriched probability space,\\ Ergodic Theory, Queueing Theory.
\section{Introduction}
\label{sec:PRELIM} 
The evolution of a number of dynamical systems depend on punctual, random perturbations which may be assumed  time-stationary.  
In such cases, the state of the system can be described, in discrete time, by a random sequence generated by a recursive, random functional termed 
\emph{driving mapping} of the recursion:
$$X_{n+1}=\varphi_n\left(X_n\right),\, n \ge 0.$$
In the general framework (of crucial interest in the application), where the sequence $\left\{\varphi_n\right\}$ driving the recursion is time-stationary but not necessarily independent, we adopt an ergodic-theoretical approach to formally address the central question of \emph{stability}, \emph{i.e.} of existence of an equilibrium state for the recursion. 

It is well-known since the pioneering works of Loynes (see \cite{Loynes62} and among others, \cite{BacBre02}), that a stationary state exists whenever 
the random maps $\varphi_n$ enjoy mild properties, such as (i) monotonicity and continuity, as assumed by Loynes, or (ii) some  
regenerative property, as in Borovkov's Theory of Renovating Events (see \cite{Foss92}). Notice that the latter framework is also suitable, under certain conditions, for random sequences that are not stochastically recursive, see \cite{Foss04}. 

However, a lot of (even very simple) models don't verify such 
assumptions. A classical example is the well-known so-called \emph{Loss queueing system}, addressed in section \ref{sec:loss}). It 
is easy to construct cases in which either none, or several stationary states may exist. For this particular model, Neveu \cite{Neu83} and Flipo 
\cite{Fli83, Fli88} have shown that the stability problem can be solved at least on a larger probability space. Their constructions,  
inspired by skew-product methods used to solve ordinary or partial differential equations, lead to  
an \emph{extension} (also called \emph{enrichment}) of the original probability space on which a stationary solution exists (see as well Lisek \cite{Lis82} for related developments). %We call such a solution, a \emph{weak} stationary solution for the recursion. 

 More recently, 
Anantharam and Konstantopoulos \cite{Anan97,Anan99} show that such extensions exist under mild assumption on the statistics of the recursion, 
using an approach based on tightness properties. The construction presented in \cite{Anan97,Anan99}, although more general, is less tractable in that the probability measure on the extension (termed \emph{weak} solution) is identified as a weak limit, and is not explicitly defined. 

Following the same directions, we aim to identify the conditions of existence of such extensions, for a more general class of models. We also propose, under such conditions, a constructive scheme of the enriched probability space - see Theorem \ref{thm:main} below.
 Our framework appears particularly adequate, when coming back to the 
original problem: it leads to several sufficient conditions of existence of a stationary state on the original probability space (see Proposition \ref{pro:appli}). Then, Loynes's Theorem and Borovkov and Foss's Theorem of Renovating events turn out to be particular cases of our result 
(see subsections \ref{subsec:loynes} and \ref{subsec:Borovkov}). As a matter of fact, the three approaches all 
rely on the same time reversal technique (usually termed \emph{Backwards scheme}). We therefore term our construction \emph{Generalized backwards scheme}. 

The outline of this paper is the following. After introducing our main notation and assumptions in section \ref{sec:prelim}, we give in 
section \ref{sec:tightness} a sufficient condition of existence of an extension solving the recursion, based on the tightness argument of Anantharam and Konstantopoulos (\emph{ibid}). The main result of this work is presented in section \ref{sec:main}: we construct explicitly the extension, and deduce several conditions for solving the original stability problem. We conclude with two cases study: in section \ref{sec:loss} we handle in this framework the stability problem of the Loss queueing system. Finally, in Section \ref{sec:impatient} we address the same problem for a generalization of this model: the Queue with impatient customers.

\section{Preliminary}
\label{sec:prelim}

Let $E$ be a Polish space that is endowed with a partial ordering $\preceq$. For all $x,y \in E$ such that $x \preceq y$, we denote  
$$\llbracket x,y \rrbracket:= \left\{z \in E; x \preceq z \preceq y\right\}.$$
We assume that $E$ admits a $\preceq $-minimal point denoted $0_{E}$, and 
is \emph{Lattice-ordered}: any $\preceq$-increasing sequence converges (possibly to some element of the adherence of $E$). 
Any subset $A \subset E$ is said 
\emph{locally finite} if for any compact subset $C \subset E$, $A \cap C$ is of finite cardinal.  
We equip $E$ with its Borel $\sigma$-field $\mathcal E$. 

Let $\Z$, $\N$ and $\N^*$ denote the sets of integers, of non-negative integers and of positive integers, respectively. 
We denote for any $x,y \in \R$, $x\vee y=\max(x,y)$, $x\wedge y=\min(x,y)$ and $x^+=x\vee 0.$

Consider a probability space $\left(\Omega,\F,\pp\right)$, furnished with the 
measurable bijective flow $\theta$ (denote $\theta^{-1}$, its measurable 
inverse). Suppose that $\pp$ is stationary and ergodic under $\theta$, 
\emph{i.e.}  
for all $\A \in \F$, $\prp{\theta^{-1}\A}=\prp{\A}$ and all $\A$ that is $\theta$-invariant (i.e. such that $\theta\A=\A$) is of  probability $0$ or $1$. Note that according to these axioms, all $\theta$-contracting event (such that $\prp{\A^c\cap\theta^{-1}\A}=0$) is 
of probability $0$ or $1$.  We denote for all $n \in \N$, 
$\theta^n=\theta\circ\theta\circ...\circ\theta\,\,\mbox{ and }\,\,\theta^{-n}=\theta^{-1}\circ\theta^{-1}\circ...\circ\theta^{-1}$.  
Except when explicitly mentioned, throughout all the random variables (r.v.'s for short) are defined on $\left(\Omega,\F,\pp\right)$. 
Under such conditions, the quadruple $\left(\Omega,\F,\pp,\theta\right)$ is termed \emph{stationary ergodic dynamical system}. 

We denote $\M(E)$ the set of measurable mappings from $E$ into itself. For any $\M(E)$-valued r.v. $F$, for any $x\in E$, let $F_{\omega}(x)$ be the image of $x$ through $F$ for the sample $\omega$. For any 
$f \in \M(E)$, and any subset $B \subset E$, we denote $f(B)=\{f(x);x\in B\}$, and accordingly 
for any $\M(E)$-valued r.v. $F$ and all sample $\omega$, $F_{\omega}(B)=\{F_{\omega}(x);x \in B\}$.  

Let $\varphi$ be a $\M(E)$-valued r.v.. 
For all $E$-valued r.v. $X$, let $\suiten{X_{X,n}}$ be the \emph{stochastic recursion} initiated by $X$ and driven by $\varphi$, i.e., such that $\pp$-a.s., 
\[\left\{\begin{array}{ll}
X_{X,0}&=X;\\
X_{X,n+1}&=\varphi\circ\theta^n\left(X_{X,n}\right),\mbox{ for all }n\in\N.
\end{array}\right.\] 
%In other words, for all $n\in\N^*$, 
%$$X_{X,n}=\varphi_{\theta^{n-1}\om}\circ\varphi_{\theta^{n-2}\om}\circ ...\circ\varphi_{\om}(X).$$ 
Define for all sample $\om$, all $n$ and $x\in E$,    
$$\Phi_{\omega}^n(x)=X_{x,n}\left(\theta^{-n}\om\right)=\varphi_{\theta^{-1}\om}\circ\varphi_{\theta^{-2}\om}\circ ...\circ\varphi_{\theta^{-n}\om}(x).$$
%$$\Psi_{\omega}^n(x)=\psi_{\theta^{-1}\om}\circ\psi_{\theta^{-2}\om}\circ ...\circ\psi_{\theta^{-n}\om}(x).$$ 
The r.v. $\Phi^n(x)$ represents the value of the recursion driven at time 0 when starting at the iteration $-n$ from the deterministic value $x$. In other words, $$\Phi_{\om}^n(x)=X_{x,n}\circ\theta^{-n}.$$ 
We investigate the existence of a stationary version of the sequence $\suiten{X_{X,n}}$, i.e. such that 
$X_{X,n}=X\circ\theta^n$ for all $n\in\N$. Then it is easily seen that the r.v. $X$ solves the functional equation  
\begin{equation}
\label{eq:recurX}
X\circ\theta=\varphi(X)\,\mbox{ a.s..}
\end{equation}
The existence of a solution to (\ref{eq:recurX}) on the original probability 
space is not granted in general, without further assumptions on $\varphi$. We aim to construct an extension of the probability space, on which a solution exists.

\section{An existence result}
\label{sec:tightness} 
Let us assume throughout this section that the couple $(\Omega,\F)$ is Polish (\emph{i.e.} $\Omega$ is Polish and $\maF$ is a sub-$\sigma$-algebra of 
the Borel $\sigma$-algebra of $\Omega$). 
Under certain conditions, 
the existence of an extension on which (\ref{eq:recurX}) admits a solution, is granted by Anantharam and Konstantopoulos's Theorem 
(see \cite{Anan97,Anan99}). This result, which identifies the probability measure on the extension as a weak limit, strongly relies  
on the property of tension of the embedded sequence of random variables. The latter holds, in particular, under the following domination assumption.  
% Throughout this section, we assume that 
\\\\
\textbf{(H1)}\emph{
%\label{hypo:1}
For some $\M(E)$-valued r.v. $\psi$, 
\begin{itemize}
\item for all $x\in E$, $0_E\preceq \varphi(x) \preceq \psi(x),\,  \P-\mbox{a.s.};$ 
\item $\psi$ is $\P$-a.s. $\preceq$-non decreasing and continuous;
\item the following recursion admits at least one $E$-valued solution: 
\begin{equation}
\label{eq:recurY}
Y\circ\theta=\psi(Y).
\end{equation} 
\end{itemize}
}
We have the following result. 
\begin{proposition}
\label{pro:anan}
Suppose that \emph{(H1)} holds, 
% that 
% \begin{itemize}
% \item[(H2)] $\psi$ is $\pp$-a.s. continuous,
% \end{itemize}
and that either one of the two following conditions holds:
\begin{itemize}
\item[\emph{\textbf{(H2)}}] $\varphi$ is $\pp$-a.s. continuous;
\item[\emph{\textbf{(H3)}}] $\varphi$ admits a.s. a finite number of discontinuities, 
and there exists a locally finite subset $0_E \in L \subset E$ that is $\pp$-a.s. stable by $\varphi$.     
\end{itemize}
Then, there exists an extension $\left(\bar\Omega, \bar\maF, \bp, \bar\theta\right)$ of 
$\left(\Omega,\maF,\pp,\theta\right)$, that is such that 
\begin{itemize}
\item $\bp$ is a $\bar \theta$-invariant probability 
on $\bar \Omega$ having $\Omega$-marginal $\pp$, 
\item there exists a $E\times \M(E)$-valued r.v. $\left(\bar X,\bar \varphi\right)$ 
defined on $\bar \Omega$ by (\ref{eq:defsolextension}), such that the $\Omega$-marginal of $\bar\varphi$ is the distribution of $\varphi$, and such that
$$\bar X\circ\bar\theta=\bar\varphi\left(\bar X\right), \bp-\mbox{a.s.}.$$
\end{itemize}
\end{proposition} 

\begin{proof} 
This result is a consequence of Theorem 1 in \cite{Anan97}, whose hypothesis are completed in \cite{Anan99}. 
Define   
%The tightness of the embedded recursion (implied by the domination assumption (H1)) is the key arguement. 
%Let $\mathcal E$ be the Borel sigma-field on $E$, and define  
\begin{itemize}
\item $\bar \Omega:=\Omega\times E$, 
\item $\bar\maF:=\maF\otimes\mathcal E,$
\item for all $(\omega,x) \in \bar \Omega$, $\bar\theta(\omega,x)=\left(\theta\omega,\varphi_{\omega}(x)\right).$ 
\end{itemize}

As an immediate consequence of Loynes's Theorem for stochastic recursions (\cite{Loynes62,BacBre02}), there exists a solution, say $Y_{\infty}$, 
to (\ref{eq:recurY}). The r.v. $Y_{\infty}$ is given by the a.s. limit of the sequence $\suiten{Y_{0_{E},n}\circ\theta^{-n}}$ (see section \ref{subsec:loynes} below).  
Note, that $Y_{\infty}$ may in general be improper (i.e. valued in some $\bar E\supset E$). However, $Y_{\infty}$ 
is $\preceq$-minimal among all the solutions of (\ref{eq:recurY}) (again, see \ref{subsec:loynes}), and the last assertion of (H1) entails that $Y_{\infty}$ is $E$-valued.  

In particular, the sequence $\suiten{Y_{0_{E},n}}$ tends weakly to $Y_{\infty}$. It is thus tight (Prohorov's Lemma): 
for all $\varepsilon>0$, there exists a compact subset $K_{\varepsilon}$ of $E$ such that for all $n \in \N$, 
$$\pr{Y_{0_{E},n} \in K_{\varepsilon}}\ge 1-\varepsilon.$$ Thus, as $E$ is Lattice-ordered, there exists  
$M_\varepsilon \in E$ s.t. 
$$\pr{Y_{0_{E},n} \preceq M_{\varepsilon}}\ge 1-\varepsilon.$$ 
In view of the first assertion of (H1), an immediate induction shows that 
$$X_{0_{E},n} \preceq Y_{0_{E},n},\,n\in\N\,\mbox{ a.s.},$$
so that %for some $M^{\prime}_{\varepsilon}$, 
$$\pr{X_{0_E,n} \preceq M_{\varepsilon}}\ge 1-\varepsilon,\,n\in\N,$$ 
which shows the tightness of $\suiten{Y_{0_E,n}}$. %and equivalently, of the sequence of probabilities 
\\

% and for all $n\in\N$, 
% $$\tilde W_n(\omega,x):=\tilde W\left(\tilde \theta^n(\omega,x)\right).$$ 
Remark that for all $n\in\N$, $\A \in \mathcal F$ and $\mathcal B \in \mathcal E$,
$$\pp\otimes\delta_{0_E}\left[\bar\theta^{-n} \left(\A \times E\right)\right]=\prp{\theta^{-n}\A}=\prp{\A}$$
and  
\begin{equation*}
\pp\otimes\delta_{0_E}\left[\bar\theta^{-n} \left(\Omega \times \mathcal B\right)\right]=\pp\otimes\delta_{0_E}\biggl[
\Bigl\{(\om,x) \in \bar\Omega; 
X_{x,n}(\om)\in\mathcal B\Bigl\}\biggl]=\prp{X_{0_E,n}\in \mathcal B}.
\end{equation*}
Hence, the probability distributions $\suiten{\left(\pp\otimes\delta_{0_E}\right)\circ \tilde \theta^{-n}}$ on $\bar\Omega$ have $\Omega$-marginal $\pp$ and $E$-marginals, the distributions of $\suiten{X_{0_E,n}}$, which form a tight sequence.   
The sequence $\suiten{\left(\pp\otimes\delta_{0_E}\right)\circ \tilde \theta^{-n}}$ is thus tight. Therefore, any sub-sequential limit 
is a good candidate for $\bp$ provided that it is $\bar\theta$-invariant. This property holds under either one of conditions (A1)-(A3) p.271-272 
in \cite{Anan99}. 
First, under condition (H2), the shift $\bar \theta$ is continuous from $\Omega\times E$ into itself, which is condition (A1) in \cite{Anan99}. 

Let us now assume that (H3) holds. Define for all $\om$, $\left\{d_j(\om)\right\}_{j\in J_{\om}},$   
the set of discontinuities of $\varphi_{\om}$ and 
$$\delta(\om)= \inf\left\{\parallel d_j(\om)-d_k(\om)\parallel; j,k \in J_{\om}\right\}.$$ 
Let us define the following events. 
$$\mathcal D=\left\{\mbox{Card }J <\infty\right\};$$
$$\mbox{For all }p\in \N^*,\,\mathcal E_p=\left\{\delta < 2^{-(p-1)}\right\}.$$ 
Note that by hypothesis, $\prp{\mathcal D}=1$, and thus $\prp{\delta>0}=1$.   
%Denote for all $x \in E$ and all $r>0$, $b(x,r)$ the open bowl of center $x$ and radius $r$. 
% 
Fix $p\in\N^*$ and a sample $\omega\in \mathcal D$. For all $j\in J_{\om}$, 
define $C_{\om,p,j}$ as follows~:
\begin{itemize}
\item[(a)] if for some $k \in J_{\om},$ $\parallel d_j(\om)-d_k(\om)\parallel \le 2^{-(p-1)}$, $C_{\om,p,j}$ is the open bowl of center 
$d_j(\om)$ and radius $\delta$;
\item[(b)] otherwise, $C_{\om,p,j}$ is the open bowl of center $d_j(\om)$ and radius $2^{-p}$, 
\end{itemize}
so that the bowls $C_{\om,p,j}$, $j\in J_{\om}$, don't intersect. 
%let  %$C_{\om,p,j}$ be some point of $E$ such that $$d\left(d_j(\om),C_{\om,p,j}\right)=2^{-p},$$ 
%$C_{\om,p,j}$ be the open bowl of center $C_{\om,p,j}$ and radius $2^{-p}$ and 
We define finally $$C_{\om,p}=\underset{j\in J_{\om}}{\bigcup}C_{\om,p,j},$$ 
and aim to construct a continuous function 
$\varphi_{\om,p}$ from $E$ into itself, such that $\varphi_{\om,p}$ coincides with $\varphi_{\om}$ outside the open set $C_{\om,p}$. 
For doing so, fix $j$ and let $x\in C_{\om,p,j}$. There exists $y$ in the frontier $\hat C_{\om,p,j}$ of $C_{\om,p,j}$ such that 
for some $\eta <1,$ $$x-d_j(\om)=\eta.\left(y-d_j(\om)\right).$$ Then, we set 
\begin{align*}
\varphi_{\om,p,j}(x)&:=2^p\parallel x- d_j(\om)\parallel.\varphi(y)\,\mbox{ in case (a)},\\
\varphi_{\om,p,j}(x)&:=\frac{1}{\delta(\om)}\parallel x- d_j(\om)\parallel.\varphi(y)\,\mbox{ in case (b)}.
\end{align*} 
The function $\varphi_{\om,p,j}$, hence radially defined, is clearly continuous on the bowl. Defining now for all $x \in E$, 
\[\varphi_{\om,p}(x)=\left\{\begin{array}{ll}
\varphi_{\om}(x)&\mbox{ if }x \notin C_{\om,p}\\
\varphi_{\om,p,j}(x)&\mbox{ if }x \in C_{\om,p,j},
\end{array}
\right.\]
we obtain the desired function.\\  
Finally, define the family of shifts $\bar\theta_p$, $p\ge 1$ for all $(\omega,x) \in \bar\Omega$ by  $$\bar\theta_p(\om,x)=\left(\theta\om,\varphi_{\om,p}(x)\right).$$ 
The $\bar\theta_p$, $p\ge 1$ are then continuous from $\bar\Omega$ into itself. 
%Fix now $x\in L$, defined in assumption (H3). 
Now fix again $p\ge 1$. It is clear from (H3) that for all $i\ge 1$, 
$\Phi^i\left(0_E\right)\in L$ a.s., hence  
\begin{equation*}
\prp{\left\{\om\,;\,\Phi^i_\om\left(0_E\right)\in C_{\om,p}\right\}\cap \mathcal D }\le 
\prp{\left\{\om\,;\,L \cap C_{\om,p} \ne \emptyset\right\} \cap \mathcal D}.
\end{equation*} 
% \begin{equation*}
% \prp{\Phi^i_{\om}(x)\in \mathcal B_p}\le 
% \prp{\min\left\{d\left(y,D^j_{\om}\right); y \in L,\,j\in J_{\om}\right\}<2^{-p}}.
% \end{equation*} 
Finally, set $$\mathcal U_p=\left\{(\om,x)\in\bar\Omega; x\in C_{\om,p}\right\}.$$ 
Then, $\mathcal U_p$ is an open subset of $\bar\Omega$, and in view of the latter inequality, 
\begin{multline*}
\underset{p \rightarrow \infty}{\lim}\underset{n \rightarrow \infty}{\lim\inf}\frac{1}{n}\sum_{i=0}^{n-1} \left(\pp\otimes\delta_x\right)\circ \bar\theta^{-i}\left(\mathcal U_p\right)\\
\begin{aligned}
&\le \underset{p \rightarrow \infty}{\lim}\underset{n \rightarrow \infty}{\lim\inf}\frac{1}{n}\sum_{i=1}^{n}
\prp{\left\{\Phi^i\left(0_E\right)\in C_p\right\}\cap \mathcal D}+\prp{\bar {\mathcal D}}\\
%\le \underset{p \rightarrow \infty}{\lim}\prp{L \cap \B_p \ne \emptyset}\\
&\le \underset{p \rightarrow \infty}{\lim}\prp{\left\{L \cap C_p \ne \emptyset\right\}\cap \mathcal D}\\
&\le \underset{p \rightarrow \infty}{\lim}\prp{\left\{\bigcup_{j\in J}L \cap C_{p,j} \ne \emptyset\right\}\cap \mathcal D\cap \left\{2^{-(p-1)}<\delta \right\}}\\
&=0,
\end{aligned}
% \end{aligned}\\
% \shoveright{+\underset{p \rightarrow \infty}{\lim}\prp{\delta \le 2^{-(p-1)}}}\\=0,
\end{multline*}
since $L$ is locally finite. Consequently, Assumption (A3) p.272 of 
\cite{Anan99} is satisfied. Hence, from Theorem 1 of \cite{Anan97}, there exists a $\bar\theta$-invariant probability $\bp$ on 
$\Omega \times \R$ whose $\Omega$-marginal
is $\pp$, given by any sub-sequential limit of  
$\suiten{\left(\pp\otimes\delta_{0_E}\right)\circ \bar \theta^{-n}}$. 

Now, define on $\bar\Omega$ the random variables   
\begin{equation}
\label{eq:defsolextension}
\bar X(\omega,x):=x,\,\,\,\bar\varphi_{\omega,x}:=\varphi_{\omega}.
\end{equation} 
We then have that
\begin{equation}
\label{eq:recurbarX}
\bar X\circ\bar\theta(\om,x)=\varphi_{\om}(x)=\bar\varphi_{\om,x}\left(x\right)=\bar\varphi_{\om,x}\left(\bar X(\omega,x)\right),\bp-\mbox{a.s.},
\end{equation}
hence $\bar X$ is a proper solution to (\ref{eq:recurX}) on $\left(\bar\Omega, \bar\maF, \bp, \bar\theta\right)$.  
\end{proof}

\section{Explicit construction}
\label{sec:main} 
%Let us now introduce our main result. 
We present the main result of this work. Under certain conditions, we can construct explicitly an extension 
solving equation (\ref{eq:recurX}). For doing so, we follow an argument related to that developed by Flipo \cite{Fli83} and Neveu \cite{Neu83} for the recursion describing the workload of a loss queueing system G/G/1/1. 

% We start with a family $\maG$ of $E$-valued random variables, or equivalently, by the random set 
% \begin{equation}
% \label{eq:defG}
% G_\om=\left\{Y(\om);\,Y \in \maG\right\}, \mbox{ a.s..}
% \end{equation}
%We assume that $G$ satisfies
We start with a random set $G$ satisfying 
\begin{equation}
\label{eq:hypoG}
\varphi_\om\left(G_\om\right) \subseteq G_{\theta\om}, \mbox{a.s.,}
\end{equation}
which is checked \emph{e.g.} by $G \equiv E$, or any deterministic set that is a.s. stable by $\varphi$. 
Now denote  for all $n\in\N^*$ 
\begin{equation}
\label{eq:defHn}
H^n_{\omega}:=\Phi_{\om}^n\left(G_{\theta^{-n}\om}\right),
\end{equation} 
the set of all possible values of the recursion driven by $\varphi$ at 
0, when letting the value at $-n$ vary over the set $G_{\theta^{-n}\om}$. Let us first remark that 

\begin{lemma}
\label{lemma:inclusion}
The sequence of random sets $\suiten{H^n}$ decreases for inclusion:   
\begin{equation}
\label{eq:inclusion}
G \supseteq H^1 \supseteq H^2 \supseteq ... \supseteq H^n \supseteq ...\,\,\,\,\,\,\mbox{ a.s.}.
\end{equation}
\end{lemma}

\begin{proof}
That 
$$H^1_\om = \varphi_{\theta^{-1}\om}\left(G_{\theta^{-1}\om}\right) \subseteq G_{\om},\mbox{ a.s.},$$ 
simply follows from (\ref{eq:hypoG}). Now, let $n\in\N^*$. We have a.s. for all $x \in H^{n+1}_{\om}$, that for some $y \in G_{\theta^{-(n+1)}\om}$,  
\begin{equation*}
%\label{eq:krush0}
x=\Phi^{n+1}_{\omega}(y)=\Phi^n_{\om}\left(\varphi_{\theta^{-(n+1)}\omega}(y)\right).
\end{equation*}  
since $y \in G_{\theta^{-(n+1)}\om}$, we have that $\varphi_{\theta^{-(n+1)}\om}(y)\in G_{\theta^{-n}\om}$ in view of (\ref{eq:hypoG}), hence 
$x \in H^{n+1}_{\om}$. 
\end{proof}
\noindent We can thus define, a.s., 
\begin{equation}
\label{eq:defH}
H_{\omega} = \lim_{n\rightarrow \infty} H^n_\om =\bigcap_{n\ge 1}H^n_{\om} \subseteq G_{\om}.
\end{equation}
% and the corresponding family $\maH$ of $E$-valued random variables, \emph{i.e.} such that
% \begin{equation}
% \label{eq:defmathcalH}
% H_\om=\left\{X(\om);\,X\in \maH\right\},\mbox{ a.s.}. 
% \end{equation}    
\begin{lemma}
\label{lemma:bijection}
Assume that (\ref{eq:hypoG}), and the following condition hold:
\begin{align}
\mbox{\emph{The random set }}& H \mbox{\emph{ defined by (\ref{eq:defH}) is such that} }\nonumber\\
&\pr{ H \mbox{\emph{ is finite and non-empty}}}>0.\label{eq:hypoH}
\end{align}
Then, the mapping $\varphi$ is bijective from $H$ to $H\circ\theta$, a.s.. The r.v. \emph{Card}$\,H$ is thus deterministic, denoted by $c$. 
\end{lemma}  
\begin{proof}
Take a sample $\om$ in the event 
$$\mathcal C:=\{H \mbox{ is finite and non-empty }\}.$$  
For any $x\in H_{\om}$, for all $n \ge 1$, there exists $y_n \in G_{\theta^{-n}\om}$ such that  
%\L$ such that $y_n \preceq Y\circ\theta^{-n}(\omega)$, 
$x=\Phi^n_{\om}(y_n)$. Therefore,  
\begin{equation*}
\varphi_{\om}(x)=\varphi_{\om}\circ\varphi_{\theta^{-1}\om}\circ...\circ\varphi_{\theta^{-n}\om}(y_n)
=\Phi^{n+1}_{\theta\om}(y_n),
\end{equation*}
where $y_n \in G_{\theta^{-n}\om}=G_{\theta^{-(n+1)}\theta\om}.$ This is true for all $n\ge 1$, hence $\varphi_{\om}(x)$ belongs to the set   
$$\bigcap_{n\ge 2}\Phi^{n}_{\theta\om}\left(G_{\theta^{-n}\theta\om}\right)= \bigcap_{n\ge 2} H^n_{\theta\om} 
%\subset \bigcap_{n\ge 1} H^n_{\theta\om}
=H_{\theta\om},$$ 
so that $\varphi_{\om}$ maps $H_{\om}$ onto $H_{\theta\om}$. Consequently,  
\begin{align}
\mathcal C & \subset \mathcal C \cap \{\varphi(H) \subset H\circ\theta\}\label{eq:cardinal0}\\
 & \subset \theta^{-1}\mathcal C.\nonumber
\end{align}
Hence, $\mathcal C$ is $\theta$-contracting, and then almost sure in virtue of (\ref{eq:hypoH}). 
So is the event on the r.h.s. of (\ref{eq:cardinal0}), thus 
$$0 < \mbox{Card }H\circ\theta \le \mbox{ Card }H <\infty, \,\mbox{ a.s..}$$ 
But $$\prp{\mbox{Card}\,H\circ\theta<\mbox{Card}\,H}>0$$ would then imply that $$\esp{\left(\mbox{Card}\,H\right)\circ\theta-\mbox{Card}\,H}<0,$$ a contradiction to the Ergodic Lemma (\cite{BacBre02}, Lemma 2.2.1).  
Therefore, $$\left(\mbox{Card}\,H\right)\circ\theta=\mbox{Card}\,H\mbox{ a.s.,}$$
which shows that Card$\,H$ is deterministic, say equal to $c$ a.s..  

Now, to check that $\varphi$ is a.s. surjective, fix a sample $\om$ on the almost sure event 
$$\left\{\mbox{Card }H=c\right\} \cap \theta^{-1}\left\{\mbox{Card }H=c\right\},$$ and $y \in H_{\theta\om}$. 
In particular, for any $n\ge 1$, for some $$x_{n+1} \in G_{\theta^{-(n+1)}\theta\om}=G_{\theta^{-n}\om},$$ we have that $$y=\Phi^{n+1}_{\theta\om}(x_{n+1})=\varphi_{\om}(y_{n}),$$ 
where $$y_{n}=\varphi_{\theta^{-1}\om}\circ...\circ\varphi_{\theta^{-n}\om}(x_{n+1})=\Phi^{n}_\om\left(x_{n+1}\right)\in H^{n}_\om.$$  
% Let $N(\om)$ be the almost surely finite integer such that $H^{n}_\om=H_\om$ for all $n\ge N(\om)$. Then, setting 
% $y_*=y_{N(\om)}$, $y_*\in H_\om$ and verifies $y=\varphi_\om(y_*)$.
Hence, $$y_n \in \bigcap_{n \ge 1} \varphi_\om\left(H^n_\om\right) = \varphi\left(\bigcap_{n\ge 1} H^n_\om\right),$$
where the last equality follows from Lemma \ref{lemma:inclusion}.   
Therefore, $\varphi_{\om}$ is surjective from $H_{\om}$ into $H_{\theta\om}$, and hence bijective since these two sets have the same 
cardinal.    
\end{proof}
We are now in position to construct an enrichment of the original probability space $\left(\Omega,\F,\pp,\theta\right)$ on which the existence of 
a solution to (\ref{eq:recurX}) is granted.  
%Denote by $c$, the almost sure cardinal of the set $H$. 
\begin{proposition} 
\label{pro:extension}
Suppose that (\ref{eq:hypoG}) and (\ref{eq:hypoH}) hold true. 
Then, the quadruple $\left(\tilde \Omega, \tilde{\F}, \qp, \tilde\theta\right)$ defines a stationary dynamical system:
%Let us then define the following extension of the space . We set 
\begin{itemize}
\item $\tilde \Omega=\left\{(\om,x) \in \Omega\times E; x\in H_\om\right\};$
\item $\tilde{\maF}$ is the trace of $\maF\otimes\mathcal E$ on $\tilde\Omega$, i.e. 
\begin{multline*}
\tilde{\maF}=\biggl\{\tilde{\maA}:=\Bigl\{(\om,x)\in\Omega\times E;\, \om\in\maA,\, x\in\maB\cap H_\om,\\
\mbox{ where }\maA\in\maF\mbox{ and }\maB\in \mathcal E\Bigl\}\biggl\}.
\end{multline*}
%$\tilde{\F}=\left\{\mathcal A\cap \tilde \Omega;  \mathcal A \in \F\otimes\mathcal P(\maK)\right\};$
\item For all $\tilde{\maA}\in\tilde{\maF}$ of the above form,  
\begin{equation*}
%\begin{aligned}
\qrp{\tilde{\maA}}%&=\frac{1}{c}\mathbf E\Bigl[\car_{\maA}\mbox{\emph{Card }}H\cap \maB\Bigl]\\
=\frac{1}{c}\int_{\Omega}\car_{\maA}(\om)\mbox{\emph{Card }}\left(H_\om\cap \maB\right)\,d\pp(\om);
%\end{aligned}
\end{equation*}
\item For all $(\om,x) \in \tilde\Omega,$ $\tilde \theta(\om,x)=\left(\theta\om,\varphi_{\om}(x)\right).$
\end{itemize}
\end{proposition}
\begin{proof}
To check this, first remark that $\tilde \theta$ defines an automorphism of $\tilde \Omega$ in view of Lemma 
\ref{lemma:bijection}. On another hand, $\qp$ defines a probability measure, since it clearly is a $\sigma$-finite measure, that is such that 
\begin{equation*}
\qrp{\tilde\Omega}=\frac{1}{c}\int_{\Omega}\car_{\Omega}(\om)\mbox{Card }\left(H_\om\cap E\right)\,d\pp(\om)=1.
\end{equation*}
Notice as well that $\qp$ has $\Omega$-marginal $\mathbf P$ since for all $\mathcal A \in \maF$, 
\begin{equation}
\label{eq:marginalP}
\qrp{\maA \times E}=\frac{1}{c}\int_{\Omega}\car_{\maA}(\om)\mbox{\emph{Card }}\left(H_\om\cap E\right)\,d\pp(\om)=\pr{\maA}.
\end{equation}
Now, fix $\tilde{\maA}:=\left\{(\om,x)\in \tilde \Omega;\, \om \in \maA,\,x\in\maB\cap H_\om\right\} \in \tilde{\maF}$. Then, 
remarking that $\tilde\theta(\om,x)\in\tilde{\maA}$ amounts to $\theta\om\in\maA$ and $\varphi_\om(x)\in\maB\cap H_{\theta\om}$, we have that 
\begin{multline*}
%\begin{aligned}
\qrp{\tilde\theta^{-1}\tilde{\maA}}
=\int\int_{\tilde\Omega}\car_{\theta^{-1}\maA}(\om)\car_{(\varphi_{\om})^{-1}(\maB\cap H_{\theta\om})}(y)\,d\qp(\om,y)\\
%&=\sum_{x\in\maB}\int\int_{\tilde\Omega}\car_{\theta^{-1}\maA}(\om)\car_{(\varphi_{\om})^{-1}(\{x\})}(y)\,d\qp(\om,y)\\
=\frac{1}{c}\int_{\Omega}\car_{\theta^{-1}\maA}(\om)\mbox{Card }\biggl((\varphi_{\om})^{-1}\left(\maB\cap H_{\theta\om}\right)
\cap H_\om\biggl)\,d\pp(\om)
.
\end{multline*}
But in view of Lemma \ref{lemma:bijection}, 
\begin{multline*}
\mbox{Card }\biggl((\varphi_{\om})^{-1}\left(\maB\cap H_{\theta\om}\right)\cap H_\om\biggl)=\mbox{Card }\biggl((\varphi_{\om})^{-1}\left(\maB\cap H_{\theta\om}\right)\biggl)\\
=\mbox{Card }\left(\maB\cap H_{\theta\om}\right),
\end{multline*}
so by $\theta$-invariance of $\pp$,
\begin{multline*}
\begin{aligned}
\qrp{\tilde\theta^{-1}\tilde{\maA}}
&=\frac{1}{c}\int_{\Omega}\car_{\maA}(\theta\om)\mbox{Card }\left(\maB\cap H_{\theta\om}\right)\,d\pp(\om)\\
&=\frac{1}{c}\int_{\Omega}\car_{\maA}(\om)\mbox{Card }\left(\maB\cap H_{\om}\right)\,d\pp(\om)\\
&=\qrp{\tilde{\maA}},
\end{aligned}
\end{multline*}
which first shows the measurability of $\tilde\theta^{-1}\tilde{\maA}$, and second, the $\tilde\theta$-invariance of 
$\qp$. The proof is complete. 
%The latter shows on the one hand, that  is measurable, and that $\qp$ is 
\end{proof}
The quadruple $\left(\tilde \Omega, \tilde{\F}, \qp, \tilde\theta\right)$ is an enrichment of $\left(\Omega, \F, \pp, \theta\right)$:  
the first space is projected onto the second one by the mapping 
\[f:\left\{\begin{array}{ll}
\tilde\Omega &\longrightarrow \Omega\\
(\omega,x) &\longmapsto \omega,
\end{array}\right.\]
and for all $\maA \in \maF$, 
$$\qp\circ f^{-1}\left[\maA\right]=\frac{1}{c}\int_{\Omega}\car_{\maA}(\om)\mbox{Card }\left(H_\om\cap E\right)\,d\mathbf P(\om)=\pr{\maA}$$
and 
$$f\circ \tilde\theta \circ f^{-1} (\maA)=\left\{f\left(\theta\om,\varphi_\om(x)\right);\om \in \maA,\,x\in H_\om\right\}=\theta\maA.$$ 

Let now $\tilde X$ (resp. $\tilde\varphi$) be the restriction 
on $\tilde \Omega$ of the r.v. $\bar X$ (resp. $\bar\varphi$) defined in (\ref{eq:defsolextension}), that is,   
$$\tilde X(\om,x)=x, \qp-\mbox{a.s.},$$
$$\tilde\varphi_{\om,x}(y)=\varphi_{\om}(y)\mbox{ for all }y \in E, \qp-\mbox{a.s.}.$$
Then, as in (\ref{eq:recurbarX}), 
\begin{equation}
\label{eq:deftildeX}
\tilde X\circ\tilde\theta=\tilde\varphi\left(\tilde X\right),\qp-\mbox{a.s.},
\end{equation}
thus $\tilde X$ is a solution to (\ref{eq:recurX}) on $\left(\tilde \Omega, \tilde{\F}, \qp, \tilde\theta\right)$. We have proven the following result.
\begin{theorem}
\label{thm:main}
If some random set $G$ satisfies (\ref{eq:hypoG}) and (\ref{eq:hypoH}),  
there exists a stationary extension 
$\left(\tilde \Omega, \tilde{\F}, \qp, \tilde\theta\right)$ of $\left(\Omega,\F,\pp,\theta\right)$, given in Proposition \ref{pro:extension}, 
on which the equation (\ref{eq:recurX}) admits a solution $\tilde X$, given by (\ref{eq:deftildeX}). 
% Moreover, $\tilde X(\omega,x) \preceq Y(\omega)$ $\qp$-almost surely on $\tilde \Omega$.  
\end{theorem}

\subsection{Resolution on the original space}
\label{subsec:resolution}
Assume in this sub-section that (\ref{eq:hypoG}) holds together with (\ref{eq:hypoH}). 
The Ergodicity of the dynamical system obtained in Theorem \ref{thm:main} is not a by-product of the  
construction, as easily understood, and similarly to that in \cite{Fli83,Neu83}. 
% In fact, it is intuitively clear that the ergodicity of the shift $\tilde\theta$ (and hence of any $\tilde\theta$-compatible sequence) depends on the structure of the state space for the particular driving mapping under consideration. % i.e on periodicity and the presence of loops, clusters, \emph{etc}... 
Notice nevertheless that the invariant sigma-field is easy to identify:  
let 
$$\mathcal I=\left\{(\om,x);\,\om\in\maA\,;x\in I_{\om}\right\}$$
be a $\tilde\theta$-invariant event of $\tilde \F$.  
Then, as 
$$\tilde\theta^{-1}\mathcal I=\left\{(\om,x);\,\om\in\theta^{-1}\maA\,; x\in \left(\varphi_\om\right)^{-1}\left(I_{\theta\om}\right)\right\},$$ 
$\mathcal I=\tilde\theta^{-1}\mathcal I$ amounts to 
\[\left\{\begin{array}{c}
\tilde\theta^{-1}\maA=\maA\\
\forall \om \in \maA,\, I_{\theta\om}=\varphi_\om\left(I_\om\right).
\end{array}\right.\] 
Then, in view of the ergodicity of $\theta$, all invariant of $\left(\tilde \Omega, \tilde{\F}, \qp, \tilde\theta\right)$ can be written, up to a $\qp$-negligible event, as 
%The $\tilde\theta$-invariant sets are of the form
\begin{equation}
\label{eq:invariant0}
\mathcal I=\left\{ (\om,x)\in \tilde\Omega\,;\, x\in I_{\om}\right\},
\end{equation}
where
\begin{equation}
\label{eq:invariant}
I_{\theta\om}=\varphi_{\om}\left(I_{\om}\right), \pp-\mbox{a.s..}
\end{equation}

A simple criterion of existence for a proper solution of (\ref{eq:recurX}) on the original space is then given in the following Lemma. 
\begin{lemma}
\label{lemma:isomorphe}
There exists a bijection between the two following sets:
\begin{multline*}
\mathcal J:=\biggl\{ E-\mbox{valued solutions of (\ref{eq:recurX}) on }\left(\Omega,\F,\pp,\theta\right)\mbox{ such that }\pr{X \in G}>0\biggl\}\\
\longleftrightarrow 
\mathcal K:=\biggl\{ \tilde\theta-\mbox{invariant sets of the form (\ref{eq:invariant0}), s.t.  Card }I=1 \mbox{ a.s.}\biggl\}.
\end{multline*}
\end{lemma}
\begin{proof}
%The proof is straightforward. 
Let $X$ be an element of $\mathcal J$, $I_{\om}=\left\{X(\om)\right\}$, a.s., 
and $$\maI=\left\{\left(\om,X(\om)\right); \om \in \Omega\right\}.$$%=\left\{(\om,x)\in \tilde\Omega\,; x\in I_{\om}\right\}.$$ 
We first have to check that $\maI \in \tilde{\maF}$, \emph{i.e.} that $X \in H$, a.s.. 
Remark that $\{X \in G\}$ is $\theta$-contracting, since a.s., whenever $X(\om) \in G_\om$, 
$$X(\theta\om)=\varphi_\om\left(X(\om)\right) \in \varphi_\om\left(G_\om\right) \subseteq G_{\theta\om}.$$ 
This event, which has a positive probability, is thus almost sure. Hence, by $\theta$-invariance 
$$\pr{\underset{n \ge 1}{\bigcap} \theta^n\left\{X \in G\right\}}=1,$$
therefore $X \in H$, a.s.. 
On another hand, $$I_{\theta\om}=\left\{X(\theta\om)\right\}=\left\{\varphi_{\om}\left(X(\om)\right)\right\}=\varphi_{\om}\left(I_{\om}\right),\mbox{ a.s.,}$$
so that $\maI \in \mathcal K$ in view of (\ref{eq:invariant}). 

Conversely, given $\maI=\left\{(\om,i(\om))\in \tilde\Omega\right\}\in\mathcal K$, it is easily seen that the r.v. defined 
on $\left(\Omega,\F,\pp,\theta\right)$ by $X(\om)=i(\om)$ is a $E$-valued solution to (\ref{eq:recurX}). Moreover, 
a.s., $$X(\om)=i(\om) \in H_{\om} \subseteq G_{\om},$$
so the r.v. $X$ is an element of $\mathcal J$. 
\end{proof}
\noindent Then, readily 
\begin{cor}
\label{cor:resol}
There exists a solution to (\ref{eq:recurX}) on the original probability space, possibly taking values in $G$, iff $\mathcal K$ is non-empty, which is the unique such solution iff $\mathcal K$ is reduced to a singleton. In particular, there exists a unique such solution 
whenever $c=1$.  
\end{cor}

We now address the closely related question of convergence of the embedded recursive sequence. Remember, that we denote for $E$-valued r.v. $X$, 
$$X_{X,n}(\om)=\varphi_{\theta^{n-1}\om}\circ ... \circ \varphi_\om\left(X(\om)\right),\,\mbox{ a.s.,}$$
the value after $n$ steps of the recursion initiated by $X$ and driven by $\varphi$.  
\begin{cor}
\label{cor:coupling}
Whenever $\mathcal K$ is non-empty, any element of $\mathcal J$ is the weak limit of some sequence 
$\suiten{X_{Y,n}}$, where $Y$ is some $E$-valued r.v. such that $Y \in G$, a.s.. 
If moreover 
\begin{equation}
\label{eq:hypoH2} 
\pr{\mbox{For some $N$ the set $H^N$ has a finite cardinal}}>0,
\end{equation}
the latter convergence holds with strong backwards coupling. 
\end{cor}

\begin{proof}
It is clear by the very definition of $H$ that for any $X \in \mathcal J$, 
for some r.v. $Y$ such that $Y\in G$, a.s., the sequence $\Phi^n\left( Y\right)=\suiten{X_{Y,n}\circ\theta^{-n}}$ converges a.s. to $X$. 
Hence, by $\theta$-invariance the sequence $\suiten{X_{Y,n}}$ tends in distribution to $X$. 

Now, on the event in (\ref{eq:hypoH2}), for any $y \in H^{N(\om)+1}_{\theta\om}$, there exists 
$x \in G_{\theta^{-N(\om)}\om}$ such that 
$$y=\varphi_\om\left(\varphi_{\theta^{-1}\om}\circ ... \circ\varphi_{\theta^{-N(\om)}\om}\left(x\right)\right),$$
so that $y \in \varphi_\om\left(H^{N(\om)}_\om\right).$ Hence, 
$$H^{N(\om)+1}_{\theta\om} \subset \varphi_\om\left(H^{N(\om)}_\om\right),$$
which implies that the event in (\ref{eq:hypoH2}) is $\theta$-contracting (taking $N(\theta\om)=N(\om)+1$), and hence almost sure whenever (\ref{eq:hypoH2}) holds.  
Therefore, in that case there exists a.s. an integer $N^\prime \ge N$ such that for all $n\ge N^\prime$, $H^n=H$. Hence, for any $Y$ as above, 
$X=\Phi^n\left(Y\right)=X_{Y,n}\circ\theta^{-n}$ a.s. for all $n \ge N^\prime$. In other words, there is strong backwards coupling between the sequences $\suiten{X_{Y,n}}$ and $\suiten{X\circ\theta^n}$ with coupling time $N^\prime$. 
\end{proof}

\subsection{Applications}
\label{subsec:appli}
\noindent We present hereafter several cases in which Theorem \ref{thm:main} applies. 
\begin{proposition}
\label{pro:appli}
Conditions (\ref{eq:hypoG}) and (\ref{eq:hypoH}) are met, and then Theorem \ref{thm:main} applies, in the following cases:
\begin{enumerate}
\item[(i)] Some deterministic finite subset $F$ of $E$ is a.s. stable by $\varphi$; 
\item[(ii)] The random map $\varphi$ is itself a.s. continuous and $\preceq$-nondecreasing;  
\item[(iii)] (H1) holds and for some solution $Y$ to (\ref{eq:recurY}),
\begin{equation}
\label{eq:regener}
\pr{Y \le 0}>0.
\end{equation}
In this case, a unique $E$-valued solution $X$ to (\ref{eq:recurX}) exists on the 
original probability space, to which all sequences $\suiten{X_{Z,n}}$, $Z \preceq Y$ a.s., converge 
with strong backwards coupling;
\item[(iv)] (H1) and (H3) hold; 
% \item[(v)] There exist a sequence of events $\left\{\maB_p;\,p \ge 1\right\}$ such that 
% $\Omega=\underset{i \ge 1}{\bigcup} \maB_i$, and a sequence 
% $\left\{\beta_i;\,i \ge 1\right\}$ of $E$-valued random variables such that for all $p \ge 1$ and $n\ge p$
% \begin{equation}
% \label{eq:v}
% \left\{X_{x,n}(\om);\,x\in E\right\} \subset  \left\{\beta_1\left(\theta^n\om\right),...,\beta_N\left(\theta^n\om\right)\right\}\,\mbox{ on }\theta^{-n}A_p.
% \end{equation} 
\item[(v)] For some collection $\maG$ of $E$-valued r.v.'s, there exists an integer $p$, a finite random set $B$ and an event $\maB$ of positive probability, such that for all $Z\in \mathcal G$ and all $n \ge p$,  
$$X_{Z,n} \in B\circ\theta^n\,\mbox{ on }\theta^{-n}\maB.$$
If additionally, $\pr{\mbox{Card }B=1}>0,$ a solution $X$ to (\ref{eq:recurX}) exists on the original space, to which all sequences 
$\suiten{X_{Z,n}}$, $Z\in\maG$, converge with strong backwards coupling.  
%where $B$ is a finite set depending only upon the sample (and not upon $x$ nor $n$). 
\end{enumerate}
\end{proposition}

\begin{proof}
\begin{itemize}
\item[(i)] Take $G = F$ a.s., so (\ref{eq:hypoG}) and (\ref{eq:hypoH}) trivially hold.
\item[(ii)] The recursion driven by $\varphi$ hence satisfies to Loynes's Theorem. See subsection \ref{subsec:loynes} below. 
\item[(iii)] Suppose that (H1) holds, and let $Y$ be an arbitrary $E$-valued solution to (\ref{eq:recurY}).  
Set $$G=\llbracket 0,Y \rrbracket\,\mbox{ a.s..}$$
Then, a.s. for all $y \in \varphi_\om\left(G_\om\right)$, 
$y=\varphi_\om(x)$ for some $x\in E$ such that $x\preceq Y(\om)$. But in view of (H1), 
\begin{equation}
\label{eq:comeback}
y\preceq \psi_\om(x)\preceq\psi_\om\left(Y(\om)\right)=Y\left(\theta\om\right),
\end{equation} 
so that $y\in G_{\theta\om}$. Hence $G$ satisfies to (\ref{eq:hypoG}). 

Now, as a consequence of Birkhoff's Theorem, (\ref{eq:regener}) implies that a.s., for some $N(\om)$, $Y\left(\theta^{-N(\om)}\om\right)=0_E.$ 
Hence, $G_{\theta^{-N(\om)}\om}=\left\{0_E\right\}$ and $$H^{N(\om)}_\om=\left\{\Phi^{N(\om)}_\om\left(0_E\right)\right\}.$$ 
Therefore, (\ref{eq:hypoH}) holds, and $c=1$. 
In particular, in view of Corollary \ref{cor:resol}, a unique solution $X$ to (\ref{eq:recurX}) exists on the original probability space, that 
is such that
$$\pr{X \in G}=\pr{X \preceq Y}>0.$$
But on the event $\{X\preceq Y\}$, again in view of (H1), 
$$X\circ\theta=\varphi(X) \preceq \psi(X) \preceq \psi(Y)=Y\circ\theta.$$
This event is thus $\theta$-contracting, and hence almost sure. This shows that $X$ is the only $E$-valued solution to (\ref{eq:recurX}). 
The strong backwards coupling property readily follows from Corollary \ref{cor:coupling}. 

\item[(iv)] Suppose now that (H3) holds additionally to (H1). Let $Y$ be a proper solution to (\ref{eq:recurY}) and $L \subset E$ be a 
locally finite subset of $E$ that is a.s. stable by $\varphi$.
Thus, as in (iii), (\ref{eq:hypoG}) is clearly met as well by  
$$G := \llbracket 0,\,Y \rrbracket \cap L\,\mbox{ a.s..}$$ 
Moreover, $G$ is a.s. of finite cardinal in view of the locally-finiteness of 
$L$, so (\ref{eq:hypoH}) holds true. 
\item[(v)] 
Set 
\begin{equation}
\label{eq:defGborovkov}
G_\om=\left\{Z(\om);\,Z \in \maG\right\},\,\mbox{ a.s.,}
\end{equation} 
and let $Z \in \maG$ and $Y=\varphi(Z)\circ\theta^{-1}$. 
Fix $n \ge p$ and $\om \in \theta^{-n}\maB$. Then, $\theta^{-1}\om \in \theta^{-(n+1)}\maB$, so 
\begin{align*}
X_{Y,n}(\om) & =\varphi_{\theta^{n-1}\om}\circ ... \circ \varphi_{\theta\om}\circ\varphi_{\om}\left(\varphi_{\theta^{-1}\om}\left(Z\left(\theta^{-1}\om\right)\right)\right)\\
& = \varphi_{\theta^{n}(\theta^{-1}\om)}\circ ... \circ \varphi_{\theta(\theta^{-1}\om)}\circ\varphi_{\theta^{-1}\om}\left(Z\left(\theta^{-1}\om\right)\right)\\
& = X_{Z,n+1}\left(\theta^{-1}\om\right)\\
& \in B_{\theta^{n+1}\left(\theta^{-1}\om\right)}=B_{\theta^{-n}\om}. 
\end{align*}
This is true on $\theta^{-n}\maB$ for all $n\ge p$, so the r.v. $Y\in \maG$. Hence, 
\begin{equation*}
%\label{eq:hypoGborovkov}
\pr{\theta\left\{\om;\,\varphi_\om\left(G_\om\right) \subset G_{\theta\om}\right\}}=1,
\end{equation*}
which amounts to (\ref{eq:hypoG}). 

It remains to check (\ref{eq:hypoH}). Let $\om \in \maB$ and $x_\om \in G_{\theta^{-n}\om}$. 
This means that for some r.v. $Z \in \maG$, $x_\om=Z\left(\theta^{-n}\om\right).$ Hence, for all $n\ge p$, as 
$\theta^{-n}\om \in \theta^{-n}\maB$, we have that 
$$
\Phi^n_\om\left(x_\om\right)=X_{Z,n}\left(\theta^{-n}\om\right)\in B_{\theta^n\left(\theta^{-n}\om\right)}=B_\om.
$$
Therefore, on $\maB$, $H \subset H^n \subset B.$ As in the proof of Corollary \ref{cor:coupling} this implies, first, that $H$ is finite on $\maB$ 
and second, that $H$ is non-empty on $\maB$ since it coincides with $H^n$ after a certain rank. 
Hence (\ref{eq:hypoH}) holds since $\maB$ is assumed to have a positive probability. 

Finally, on the event $\{\mbox{Card }B=1\}$, $H=H^n=B$ for all $n\ge p$, so $c=1$. Whenever this event is of positive probability, 
the latter is true a.s., so $c=1$ and we can set $H_\om=\{X(\om)\}$, a.s.. Once again, the strong backwards coupling to $X$ follows 
from Corollary \ref{cor:coupling}.  
% Set again $G=E$ a.s. and take $\om \in \maB$. Then, $\theta^{-n}\om \in \theta^{-n}\maB$, so for all $x \in E$ and all $n\ge N$ 
% $$\Phi^n_\om\left(x\right)=X_{x,n}\left(\theta^{-n}\om\right)\in B_\om.$$
% The latter means that on $\maB$, the sets $H^n$, $n\ge N$ (which are a.s. non-empty by construction) are all included in $B$, and hence finite. 
% Thus, on $\maB$ there exists a finite (random) $N''\ge N$ such that $H^n=H$ for all $n\ge N''$, so $H$ is non-empty and of finite cardinal. 
% But $\pr{\maB}>0$, hence (\ref{eq:hypoH}). 
%% \item[(v)] Almost surely, $\om \in A_{N(\om)}$ for some $N(\om) \ge 1$. Therefore, for all $n\ge N(\om)$, 
%% $\theta^{-n}\om \in \theta^{-n}A_{N(\om)}$. Therefore, in view of (\ref{eq:v}), for all $x\in E$ 
%% $$\Phi^n_\om\left(x\right)=X_{x,n}\left(\theta^{-n}\om\right)\in \left\{\beta_1(\om),...,\beta_{N(\om)}(\om)\right\}.$$
%% Taking $G=E$ a.s., this means that  
%% $$H \subset H^n \subset \left\{\beta_1,...,\beta_{N}\right\}\,\mbox{ a.s.},$$
%% hence (\ref{eq:hypoH}).
\end{itemize}
\end{proof}

Whenever (H1) holds together with (H3), (iv) provides an alternative proof 
of Proposition \ref{pro:anan}.  In fact, by $\tilde\theta$-invariance of $\qp$, the sequence 
of probability $\suiten{\qp\circ\tilde\theta^{-n}}$ is tight since it is constant. So replacing $\pp\otimes\delta_{0_{E}}$ by $\qp$ (which 
has $\Omega$-marginal $\mathbf P$ - see (\ref{eq:marginalP}))
in the proof of Proposition \ref{pro:anan} would lead to the same extension $\left(\tilde \Omega, \tilde{\F}, \qp, \tilde\theta\right)$.

\subsection{On Loynes's Theorem}
\label{subsec:loynes}
Our construction allows us to capture Loynes' celebrated Theorem for monotonic recursions (\cite{BacBre02,Loynes62}). 
We assume here that $\varphi$ is a.s. non-decreasing and continuous on $E$. Loynes's sequence is classically defined as 
$\suiten{\Phi^n\left(0_E\right)}$. 
It is routine to check that the latter is a.s. nondecreasing. Let a.s., $Y$, its supremum, that we assume to be $E$-valued. 
Then, by continuity, 
\begin{equation}
\label{eq:recurY2}
\varphi(Y)=Y\circ\theta\mbox{ a.s..}
\end{equation}
%so $Y$ provides a solution to (\ref{eq:recurX}). 
%
We set  
$$G=\llbracket 0,Y \rrbracket\,\mbox{ a.s..}$$ 
Let $n\ge 1$. From (\ref{eq:recurY2}), we have that 
$$\Phi^n_\om\left(Y\left(\theta^{-n}\om\right)\right)=Y(\om)\mbox{ a.s.}.$$ 
Therefore, since $\Phi^n$ is a.s. non-decreasing (as easily seen by induction), 
\begin{align*}
H^n_\om &=\left\{\Phi^n_\om(x);\, x \in \left\llbracket 0_E ;\, Y\left(\theta^{-n}\om\right) \right\rrbracket \right\}\\
&= \big\llbracket \Phi^n_\om\left(0_E\right);\,\Phi^n_\om\left(Y\left(\theta^{-n}\om\right)\right)\big\rrbracket\\
&= \left\llbracket \Phi^n_\om\left(0_E\right);\,Y(\om)\right\rrbracket\mbox{ a.s..}
\end{align*}
As $Y$ is the a.s. limit of Loynes's sequence, it readily follows that 
$$H=\left\{Y\right\}\mbox{ a.s..}$$
Using Corollary \ref{cor:resol}, we obtain that the only solution $Z$ to (\ref{eq:recurX}) on 
the original space such that $\pr{Z \le Y}>0$ is the the r.v. $Y$ itself. Thus $Y$ is the a.s. minimal solution, which 
is the exact statement of Loynes's Theorem.

\subsection{Renovating events}
\label{subsec:Borovkov}
Condition (v) of Proposition \ref{pro:appli} can be rephrased in the following comprehensive terms: 
whatever the initial r.v. $X_0=X$ in a given collection, after a deterministic rank $N$, the recursion is valued with positive probability 
in a finite range depending only upon the sample.  We will give in Section \ref{sec:loss} a concrete application of this result, which is, clearly, a generalization of the concept of \emph{Renovating events} 
(see \cite{Foss92} and \cite{BacBre02}, p.115). In fact, we have expressed condition (v) in the form that better emphasizes this connexion. This 
will allow us to show readily that the typical existence and coupling result of Renovating events theory (Corollary 2.5.1 in \cite{BacBre02}) is in fact a particular case of (v) of Proposition \ref{pro:appli}.  
% in the most obvious fashion. 

Let us briefly recall 
% For this, it might be usefull to rewrite the recursion $\suiten{X_n}$ driven 
% by $\varphi$ in terms of an $F$-valued r.v. $\gamma$ (where $F$ is some auxilliary space) and the deterministic mapping 
% \[f:\left\{\begin{array}{ll}
% E\times F &\to E\\
% (x,\gamma(\om)) &\mapsto f\left(x,\gamma(\om)\right)=\varphi_\om(x).
% \end{array}\right.\]
that a stationary sequence of events $\suiten{\theta^{-n}\maA}$ 
(where $\maA$ is of positive probability) is termed sequence of renovating events of length $m \in \N^*$ for the 
recursion $\suite{X}$ whenever for some $E^{\prime}$-valued r.v. $\beta$ (where $E^{\prime}$ is some auxilliary space), some deterministic mapping 
$\Psi:\left(E^{\prime}\right)^m\rightarrow E$, for all $n\ge m$,  
\begin{equation}
\label{eq:borovkov}
X_n=\Psi\left(\beta\circ\theta^{n-m},...,\beta\circ\theta^{n-2},\beta\circ\theta^{n-1}\right)\,\mbox{ on $\theta^{-(n-m)}\maA$.}
\end{equation}
Now let $\mathcal Z$ a collection of r.v.'s, for which we assume that all sequences $\suiten{X_{Z,n}}$, $Z\in \mathcal Z$, admit the same sequence 
of renovating events 
$\suiten{\theta^{-n}\maA}$, with the same length $m$ and same function $\Psi$. 
It is then straightforward that (v) holds. Take indeed $\maG:=\mathcal Z$, $\maB:=\theta^m\maA$ and $p:=m$. Then, 
for all $n \ge p$, $\theta^{-n}\maB=\theta^{-(n-m)}\maA$, so on this event,  
$$X_{Z,n}(\om)=\Psi\left(\beta\circ\theta^{n-m},...,\beta\circ\theta^{n-2},\beta\circ\theta^{n-1}\right)\mbox{ for all }Z\in \mathcal Z.$$ 
Therefore, condition (v) is satisfied when taking 
$$B=\left\{\Psi\left(\beta\circ\theta^{-m},...,\beta\circ\theta^{-2},\beta\circ\theta^{-1}\right)\right\},\,\mbox{ a.s..}$$
In particular, $c=1$, so there is a unique solution to (\ref{eq:recurX}) on the original probability space, to which all sequences 
$\suiten{X_{Z,n}}$, $Z\in \mathcal Z$ converge with strong backwards coupling. This is Borovkov and Foss's Theorem (see 
\cite{Foss92} and \cite{BacBre02}, Corollary 2.5.1).

\section{The Loss Queue}
\label{sec:loss}
The classical, but challenging problem of finding the stability region of the \emph{Loss Queue} G/G/1/1 can be addressed in our 
framework. 
Consider a queueing system having one server and no waiting room, so that each customer is either immediately served (if the system is empty), or 
rejected upon arrival (if the server is busy). We assume that the input in this queue is of the \emph{G/G} type, and work   
on the Palm space $(\Omega,\maF,\pp,\theta)$ of  the arrival process $$... < T_{-2} < T_{-1} < T_0=0 < T_1 <...\,,$$
where $T_n$ is the arrival time of the $n$th customer, denoted $C_n$. 
The stationary sequence of inter-arrival times $\suitez{\xi}:=\suiten{T_{n+1}-T_n}$ is then compatible with $\theta$, i.e. $\xi_n=\xi\circ\theta^n$ 
for all $n$. The service times $\suitez{\sigma}$ requested by the customers form a sequence of marks of the arrival process, which is hence as well compatible with $\theta$ (see \emph{e.g.} \cite{BacBre02} for the Ergodic-theoretical representation of stationary queueing systems). 
We denote $\sigma$ the generic service time, and assume that $\sigma$ is a.s. non-negative, and $\xi$ is a.s. positive.

Let $W_n$ be the workload (i.e. the quantity of work in the system, in time unit) seen by $C_n$ upon arrival. As easily checked, 
the workload sequence is a $\R+$-valued recursion driven by the random map
$$\varphi_\om(x)=\left[x+\sigma(\om)\car_{\{x=0\}}-\xi(\om)\right]^+,$$ 
which is not a.s. non-decreasing and monotonic. Despite its simplicity, this model can not be handled by Loynes's framework. As a matter of fact, 
it is quite simple to exhibit examples for which uniqueness, and even existence of a solution to (\ref{eq:recurX}) don't hold 
(see \cite{BacBre02}, p.121 - and the examples hereafter). 
The existence of a stationary workload defined on $\Omega\times\N$, and a constructive scheme are presented in \cite{Fli83,Neu83}, 
whereas the existence on $\Omega\times\R^+$ is proven in \cite{Anan97,Anan99} using the tightness approach, as developped in section \ref{sec:tightness}. Hereafter, we use Theorem \ref{thm:main} to construct explicitly this solution, and relate it to those of \cite{Fli83,Neu83}. 

First, denote a.s.
\begin{align}
A&=\left\{i>0\,;\,\sigma\circ\theta^{-i}-\sum_{j=1}^i\xi\circ\theta^{-j}>0\right\};\label{eq:defAloss}\\
\gamma &=\sup A.
\end{align}
The set $A$ thus contains all the absolute values of indexes of the customers possibly in the system at time 0, which are those who found an empty system upon arrival, and did not complete their service at 0. 

Remark, that 
\begin{lemma}
\label{lemma:majorerho}
The r.v. $\gamma$ is a.s. finite. In particular, there exists an integer $g$ such that
\begin{align}
g&=\min\left\{ n>0 ; \pr{\gamma \le n}>0\right\}\label{eq:majorerho}.
\end{align}
\end{lemma}
\begin{proof}
It is a consequence of Birkhoff's Theorem that $$\sigma\circ\theta^{-n}-\displaystyle\sum_{j=1}^n\xi\circ\theta^{-j} \tend -\infty\mbox{ a.s.,}$$  
so there exists a.s. $N<+\infty$ such that the latter expression is non-positive for all $i \ge N$. In particular, a.s. $\gamma \le N <+\infty$. 
For any $n$ such that $\pr{N=n}>0$ (such integers exists since $N < +\infty$ a.s.), $\pr{\gamma \le n}>0$, so $g$ is well-defined. 
\end{proof}

In view of the above remark, on the event $\{\gamma \le g\}$ the workload at time 0 is an element of the set  
$$B:=\left\{ \sigma\circ\theta^{-i}-\sum_{j=1}^i\xi\circ\theta^{-j}\,;\,i=1,...,g\right\}.$$
In other words, for any $E$-valued r.v $Z$ and for all $n\ge g$, $\Phi^n\left(Z\circ\theta^{-n}\right) \in B$ on $\{\gamma \le g\}$, that is to say 
$$W_{Z,n}\in B\circ\theta^n\,\mbox{ on }\theta^{-n}\{\gamma \le g\}.$$
We are thus in the case (v) of Proposition \ref{pro:appli} taking $G:=\R+$ a.s., $p:=g$ and $\maB=\{\gamma \le g\}$. 
Theorem \ref{thm:main} thus applies to the workload sequence: there exists 
an extension $\left(\tilde \Omega, \tilde{\F}, \qp, \tilde\theta\right)$ on which (\ref{eq:recurX}) admits a solution. 

We aim to compare our extension to that presented in \cite{Fli83}. 
Let us briefly recall 
the construction proposed therein. 
Define almost surely, for all $i\in\N$, 
\[\ell_\om(i)=\left\{\begin{array}{ll}
i+1&\mbox{ if $C_{-i}$, provided he found an empty system upon arrival,}\\
   &\mbox{ is still in service at $T_0-$,}\\
0 &\mbox{ else,}
\end{array}\right.\] 
and for all $n\ge 1$, 
$$L^n_{\om}(i)=\ell_{\theta^{-1}\om}\circ\ell_{\theta^{-2}\om}\circ ...\circ\ell_{\theta^{-n}\om}(i).$$
In words, $L^n(i)$ represents the index of the customers present in the system at $T_0-$ when assuming that customer 
$C_{-n-i}$ found an empty system upon arrival. 
Denoting then $\hat H^n_\om=L^n_\om(\N)$ and $\hat H_\om=\bigcap_{n\ge 1}\hat H^n_\om$, one can show (see \cite{Fli83}), as in Lemma \ref{lemma:bijection}, that $\hat H$ is an a.s. finite subset of $\N$ having a deterministic cardinal. Hence, an enrichment   
 $\left(\hat \Omega, \hat{\F}, \hp, \hat\theta\right)$ exists, that is defined similarly to that in Proposition \ref{pro:extension}, 
replacing $\varphi$ by $\ell$ and $H$ by $\hat H$.  
Moreover, (\ref{eq:recurX}) is solved on this extension by setting 
$$\hat X(\om,i)=\left[\sigma\left(\theta^{-i}\om\right)-\sum_{j=1}^i\xi\left(\theta^{-j}\om\right)\right]^+,\hp-\mbox{a.s.},$$
$$\hat\varphi_{\om,i}=\varphi_\om,\hp-\mbox{a.s.}.$$  

As shown in the next Lemma, $\left(\hat \Omega, \hat{\F}, \hp, \hat\theta\right)$ can be projected onto the extension 
$\left(\tilde \Omega, \tilde{\F}, \qp, \tilde\theta\right)$ constructed by Proposition \ref{pro:extension}.    
% the extension $\left(\tilde \Omega, \tilde{\F}, \qp, \tilde\theta\right)$ constructed by our method in that particular case, is thiner than 
% $\left(\hat \Omega, \hat{\F}, \hp, \hat\theta\right)$:
\begin{lemma}
\label{lemma:compaFlipo}
The following mapping is a.s. surjective:
\[F_\om:\left\{\begin{array}{ll}
\hat H_\om &\longrightarrow H_\om\\
i &\longmapsto \Phi^i_\om(0),
\end{array}\right.\]
where $\Phi^0_\om(0)$ is naturally set to $0$. 
\end{lemma}
\begin{proof}
Fix a sample $\om$, and let us first check that $F_\om$ maps $\hat H_\om$ onto $H_\om$. 
Let $j\in \hat H_\om$. For all $n\ge 1$, there exists $i_n\in \N$ such that $j=L_\om^n(i_n)$. In other words, for the sample $\om$,  
$C_{-j}$ is in service just before time $T_0$ whenever $C_{n+i_n}$ entered an empty system, hence  
$$\Phi^{n+i_n}_\om(0)=\Phi_\om^j(0)=F_\om(j).$$ 
Therefore, $F_\om(j)\in\Phi^{n+i_n}_\om(\R+)$, %\left(\left[0,Y\left(\theta^{-(n+i_n)}\om\right)\right]\right),$$
so there exists $n^\prime=n+i_n\ge n$ such that $F_\om(j)\in H^{n^\prime}_\om$. This is true for all $n\ge 1$, 
hence $F_\om(j)\in H_\om$.  

Now, to check that $F_\om$ is surjective, take $x\in H_\om$ and let for all $n\ge 1$, 
$x_n\in \left[0,Y\left(\theta^{-n}\om\right)\right]$ be such 
that $x=\Phi^n_\om\left(x_n\right)$. 
First, as shown above, there exists $j\in \{0,1,...,\gamma\}$ such that $$x=\Phi_\om^j(0)=F_\om(j).$$ 
Fix now $n\ge 1$. Then, assuming that for all $\tilde n\ge n$, 
$$x_{\tilde n}(\om)-\sum_{j=n+1}^{\tilde n}\xi\left(\theta^{-j}\om\right) \ge 0$$ 
would %imply that $$Y\left(\theta^{-\tilde n}\om\right)-\sum_{j=n+1}^{\tilde n}\xi\left(\theta^{-j}\om\right)\ge 0,$$ 
contradict Birkhoff's Theorem (remember that $\esp{\xi}>0$). Then, there exists 
$\tilde n \ge n$ such that $x_{\tilde n}(\om)-\sum_{j=n+1}^{\tilde n}\xi\left(\theta^{-j}\om\right) <0,$ which means that either (i) 
$x_{\tilde n}=0$ and the system was empty upon the arrival of $C_{-\tilde n}$ or (ii) $C_{-\tilde n}$ found a busy server 
upon arrival, having a residual workload of $x_{\tilde n}$, and the customer in service at that instant has left the system 
before the arrival of $C_{-n}$. In both cases, whenever $C_{-\tilde n}$ found a workload equal to $x_{\tilde n}$ upon 
arrival, there exists an index $\hat n \in \left\{n,n+1,...,\tilde n\right\}$ such that the system is empty 
at the arrival of $C_{-\hat n}$. In other words, $\Phi^{\tilde n}_\om\left(x_{\tilde n}\right))=\Phi^{\hat n}_{\om}(0)$.  

As a consequence, there exists a non negative integer $i_n:=\hat n-n$ such that 
$$\Phi_\om^j(0)=x=\Phi^{\tilde n}_\om\left(x_{\tilde n}\right)=\Phi^{n+i_n}_{\om}(0),$$ 
which amounts to say that $j=L_\om^n(i_n)$. This is true for all 
$n\ge 1$, hence $j \in \hat H_\om$, which concludes the proof.  
\end{proof}

We now introduce two simple examples (given in \cite{BacBre02}, p.122), in which existence or uniqueness of a stationary workload don't hold on the original probability space, and address the stability problem in our framework. 
We work on the following elementary ergodic dynamical system:
\[\left\{\begin{array}{ll}
\Omega &=\left\{\om_1,\om_2\right\};\\
\maF &=\mathcal P(\Omega);\\
\mathbf P &:= \mbox{ uniform on }\Omega;\\
\theta &:\om_1 \longleftrightarrow  \om_2.
\end{array}\right.\]

\noindent\textbf{Example 1}\\
\noindent
Set, say, 
\[\left\{\begin{array}{ll}
\xi (\om_1) &=\xi(\om_2)=1;\\
\sigma(\om_1) &=1,\,\sigma (\om_2)=:y>2.
\end{array}\right.\] 
We will only treat in detail the case where $y \not\in \N$ and $\lfloor y \rfloor$ is an odd number. The other cases are analogous. 
Then, readily 
\begin{align*}
A_{\om_1}&=\left\{1,3,5,...,\lfloor y \rfloor\right\},\,\gamma(\om_1)=\lfloor y \rfloor;\\
A_{\om_2}&=\left\{2,4,6,...,\lfloor y \rfloor-1\right\},\,\gamma(\om_2)=\lfloor y \rfloor -1.
\end{align*} 
Let $i \in A_{\om_1}$ and fix $n\ge 1$. Then, it is always possible to find an $x\in \R+$ such that when assuming that the recursion 
equals $x$ at the arrival of customer $C_{-n}$, $C_{-i}$ is in service at time 0. Indeed,
\begin{itemize}
\item If $n$ is odd, 
\begin{itemize}
      \item If $n \equiv i  \mbox{ mod.}\left(\lfloor y \rfloor +1\right)$, say $n=i-2 + p\left(\lfloor y \rfloor +1\right)$, set 
      $$W_{-n}:=x=0.$$ Then $C_{-n}$ is served, and 
      \begin{multline*}
            W_{-\left(i+(p-1)\left(\lfloor y \rfloor +1\right)+1\right)}\\
            \begin{aligned}
               &=y -\Bigl\{\left(i+p\left(\lfloor y \rfloor +1 \right)\right) 
                 - \left(i+(p-1)\left(\lfloor y \rfloor +1 \right)+1\right)\Bigl\}\\
               &=y-\lfloor y \rfloor>0, 
            \end{aligned}   
            \end{multline*}
      whereas 
      $$W_{-\left(i+(p-1)\left(\lfloor y \rfloor +1\right)\right)}=\left[y-\left\{\lfloor y \rfloor +1\right\}\right]^+ = 0.$$      
      Therefore, $C_{-\left(i+(p-1)\left(\lfloor y \rfloor +1\right)\right)}$ is served and by an immediate induction, all the customers $C_{-k}$,    where $k \in \{i,....,n-1\}$ and $k\equiv i \mbox{ mod.}\left(\lfloor y \rfloor +1\right)$ are served. 
      In particular, $C_{-i}$ 
      is served, and is still in the system at 0. So 
      $$\Phi^n_{\om_1}(x)= \sigma\circ\theta^{-i}(\om_1)-\sum_{j=1}^i \xi\circ\theta^{-j}(\om_1)=y-i;$$
      %\item[$\vdots$] 
      \item If $n \equiv i-2\ell\,  \mbox{ mod.}\left(\lfloor y \rfloor +1\right)$ (say $n=i-2\ell + p\left(\lfloor y \rfloor +1\right)$), set 
      $$W_{-n}:=x \in \Bigl(\lfloor y \rfloor -2\ell,\lfloor y \rfloor -2\ell +1\Bigl].$$ 
      Then, \begin{multline*}
            W_{-\left(i+(p-1)\left(\lfloor y \rfloor +1\right)+1\right)}\\
            \begin{aligned}
               &=x-\Bigl\{\left(i-2\ell+p\left(\lfloor y \rfloor +1 \right)\right) - \left(i+(p-1)\left(\lfloor y \rfloor +1 \right)+1\right)\Bigl\}\\
               &=x-\left\{\lfloor y \rfloor -2\ell\right\} >0,
            \end{aligned}   
            \end{multline*}
      whereas 
      $$W_{-\left(i+(p-1)\left(\lfloor y \rfloor +1\right)\right)}=\left[x-\left\{\lfloor y \rfloor -2\ell +1\right\}\right]^+= 0,$$
      so $C_{-\left(i+(p-1)\left(\lfloor y \rfloor +1\right)\right)}$ is served, and as above, all $C_{-k}$ where\\ 
      $k\in \left\{i,....,\left(i+(p-1)\left(\lfloor y \rfloor +1\right)\right)\right\}$ and $k\equiv i \mbox{ mod.}\left(\lfloor y \rfloor +1\right)$ 
      are served. Here again, $C_{-i}$ is thus served and $\Phi^n_{\om_1}(x)=y-i.$ 
\end{itemize}
\item if $n$ is even, set $W_{-n}=x+1$ for the different values of $x$ set above, so that $W_{-(n-1)}=x$ for $n-1$ odd and we can apply the above 
argument.
\end{itemize}
Therefore, in any case, for all $i \in A_{\om_1}$ and all $n$, there exists $x\in\R+$ such that 
$\Phi^n_{\om_1}(x)=y-i$. This shows that $y-i \in H_{\om_1}$ for all such $i$. On the other hand, provided some customer $C_{-n}$ is served, where $n$ is even and $n> \lfloor y \rfloor$, the service time of $C_{-n}$ is 1, so $C_{-(n-1)}$ is served as well. Then $n-1$ is odd with 
$n \equiv i-2\,  \mbox{ mod.}\left(\lfloor y \rfloor +1\right)$ for some $i \in A_{\om_1}$, so as above $C_{-i}$ is in service at 0. 
In particular, there is \emph{always} a customer in service at time 0. 

Therefore, for all $n\ge \lfloor y \rfloor +1$, 
$$\left\{y-i\,;\, i \in A_{\om_1}\right\} \subset H_{\om_1} \subset \Phi^n_{\om_1}\left(\R+\right) \subset \left\{y-i\,;\, i \in A_{\om_1}\right\},$$
hence
$$H_{\om_1}=\left\{y-1,y-3,y-5,...,y-\lfloor y \rfloor\right\}.$$
Analogously, we can check that
$$H_{\om_2}=\left\{y-2,y-4,y-6,...,y-\left(\lfloor y \rfloor -1\right),0\right\}$$
(indeed the system may be empty at 0 for the sample $\om_2$ whenever $C_{-\lfloor y \rfloor +1}$ is served). In particular, 
$c=\frac{\lfloor y\rfloor +1}{2}.$ 

Now notice that $\varphi_{\om_1}(y-i)=\left[y-(i+1)\right]^+$ for all odd $i$ such that $i \le \lfloor y \rfloor$, whereas  
$\varphi_{\om_2}\left(y-(i+1)\right)=y-(i+2)$ for all odd $i$, $i<\lfloor y \rfloor$ and $\varphi_{\om_2}(0)=y-1$. As a conclusion, recalling 
(\ref{eq:invariant0}) we easily check that the invariant sigma-field of $\left(\tilde \Omega, \tilde{\F}, \qp, \tilde\theta\right)$ is 
$\left\{\emptyset,\tilde\Omega\right\}$. In particular, the set $\mathcal K$ of Lemma \ref{lemma:isomorphe} is empty: there is no solution on the original probability space.  
\textit{ }\\\\\\
\noindent\textbf{Example 2}\\
\noindent
On the same probability space, define now 
\[\left\{\begin{array}{ll}
\xi (\om_1) &=\xi(\om_2)=1;\\
\sigma(\om_1) &=:x,\,\sigma (\om_2)=:y,
\end{array}\right.\] 
where $x$ and $y$ both belong to the open interval $(1,2)$. 
Following the same lines as in Example 1, it is easily seen that 
$$A_{\om_1}=\{1\}\,;\,A_{\om_2}=\{1\},$$
and 
$$H_{\om_1}=\{0,y-1\}\,;\,H_{\om_2}=\{0,x-1\}.$$
It is then immediate that both events 
\begin{align*}
\mathcal I &=\left\{(\om_1,0)\,;\,(\om_2,x-1)\right\},\\
\mathcal I' &=\left\{(\om_1,y-1)\,;\,(\om_2,0)\right\}\\
\end{align*}
belong to $\mathcal K$ in this case. There are two solutions to (\ref{eq:recurX}).

\section{The Queue with impatient customers}
\label{sec:impatient} 
We now consider a queueing model with impatient customers. We use the same notation and assumptions as in section \ref{sec:loss},   
except that customer $C_n$ now requires to enter service before a given deadline, say at $T_n+D_n$. If not, the customer leaves the system 
at $T_n+D_n$ and is lost forever. 
We consider, that as soon as a customer entered the service booth, his service will proceed without interruption even though 
his patience elapses during his service. We assume that $\suitez{D}$ is a sequence of marks of the arrival process, and that 
the generic r.v. $D$ is non-negative. 
%Suppose as well that the generic inter-arrival $\xi$ is not a.s. null. 
 The system has  
a single server, operating in the order of arrivals (FIFO). Then (see \cite{MR86d:60103,HebBac81,Moy09}), the workload sequence 
$\suitez{X}$ is stochastically recursive, driven by the mapping 
$$\varphi_\om(x)=\left[x+\sigma(\om)\car_{\{x\le D(\om)\}}-\xi(\om)\right]^+,$$
since each given customer is proposed a waiting time before entering service, that equals the workload just before his arrival time. 
Hence, the loss queue is a particular case of this model for $D=0$ a.s., as easily understood. 
We aim once again to solve 
\begin{equation}
\label{eq:recurIMP}
Y\circ\theta=\varphi\left(Y\right),\mbox{ a.s.}.
\end{equation} 

As can easily be checked, we have a.s. for all $x$, 
\begin{equation}
\label{eq:compphi}
\chi_{\om}(x) \le \varphi_\om(x) \le \psi_{\om}(x),
\end{equation}
where 
\begin{align*}
\chi(x)&=\left[x\vee \left(\sigma \wedge D\right) -\xi\right]^+,\\
\psi(x)&=\left[x\vee \left(\sigma + D\right) - \xi\right]^+
\end{align*}
(see eq. (9) and (12) in \cite{Moy09}). 

The random maps $\chi$ and $\psi$ are a.s. continuous and non-decreasing, and 
the only proper solutions $Y$ and $Z$ to the recursions respectvely driven by $\chi$ and $\psi$ read  
\begin{align}
Y&=\left[\max_{i\ge 1}\left(\left(\sigma\wedge D\right)\circ\theta^{-i}-\sum_{j=1}^i \xi\circ\theta^{-j}\right)\right]^+,\label{eq:defY}\\
Z&=\left[\max_{i\ge 1}\left(\left(\sigma+D\right)\circ\theta^{-i}-\sum_{j=1}^i \xi\circ\theta^{-j}\right)\right]^+.\label{eq:defZ}
\end{align}
Then, if $\pr{Z=0}>0$, we are in the configuration of (iii) of Proposition \ref{pro:appli}, and a unique 
solution exists on the original probability space.  

If $\pr{Z=0}=0$, a construction based on tightness 
arguments is proposed in \cite{Moy09}, that establishes the existence of a stationary workload on $\Omega \times \R+$ provided 
that $\sigma$ and $\xi$ both take value in a set of the form 
\begin{equation}
\label{eq:defLalpha}
L_\alpha:=\left\{n\alpha;\,n\in \N\right\},\mbox{ where }\alpha\in \R+.
\end{equation}
In fact, Theorem \ref{thm:main} applies, and we can explicitly construct the extension in this case. 
Indeed, 
clearly $\varphi(x)\in L_\alpha$ a.s. for all $x\in L_\alpha$, so that 
the recursion, when initiated in $L_{\alpha}$, remains in this set forever. We are thus in the case (iv) of Proposition 
\ref{pro:appli} taking $L=L_{\alpha}$. More precisely,  
as in (\ref{eq:comeback}) and in view of 
(\ref{eq:compphi}), we have a.s. that for any 
$Y \le x \le Z$,  
$$Y\circ\theta = \chi(Y) \le \chi(x) \le \varphi(x) \le \psi(x) \le \psi(Z) = Z\circ\theta. $$ 
Hence, the random set defined by
\begin{equation}
\label{eq:defGIMP}
G = L_\alpha \cap \left[Y, Z\right]\,\mbox{ a.s.} 
\end{equation}
satisfies to (\ref{eq:hypoG}). Moreover, $Z$ is a.s. finite (see Lemma \ref{lemma:majoretau} below), thus  
$G$ is a.s. of finite cardinal. So does $H$: (\ref{eq:hypoH}) holds true and Theorem \ref{thm:main} applies. 

\subsubsection*{Size of the extension}
Let us investigate more precisely the form of the extension. First, denote   
\begin{equation}
\label{eq:defs}
\underaccent{\bar}{s} = \min\left\{n \in \N;\,\pr{\sigma \le n\alpha}>0\right\},
\end{equation}
\begin{equation}
\label{eq:defS}
\bar s = \inf\left\{n \in \N;\,\pr{\sigma \le n\alpha}=1\right\},
\end{equation}
\begin{equation}
\label{eq:defD}
\bar d = \inf\left\{n \in \N;\,\pr{D \le n\alpha}=1\right\},
\end{equation}
where $\bar s$ and $\bar d$ may be set to $+\infty$. 
Denote, a.s., 
\begin{align}
A&=\left\{i>0\,;\,\left(D\circ\theta^{-i}\right)-\sum_{j=1}^i\xi\circ\theta^{-j}>0\right\}\mbox{ and }\tau^-=\sup A;\nonumber\\
\tau^+&=\min\left\{i>0\,;\, \sum_{j=0}^{i-1} \xi\circ\theta^{j} \ge D\right\};\nonumber\\
B&=\left\{i>0\,;\,\left(\sigma\circ\theta^{-i}+D\circ\theta^{-i}\right)-\sum_{j=1}^i\xi\circ\theta^{-j}>0\right\}\mbox{ and }\rho=\sup B\label{eq:defrho}.
% C&=\left\{i>0\,;\,\left(\sigma\circ\theta^{-i}\wedge D\circ\theta^{-i}\right)-\sum_{j=1}^i\xi\circ\theta^{-j}>0\right\}\mbox{ and }\nu^-=\sup C;\nonumber\\
% \nu^+&=\min\left\{i>0\,;\, \sum_{j=0}^{i-1} \xi\circ\theta^{j} \ge \sigma\wedge D\right\}.
\end{align}
Each customer spends in the waiting line (resp. in the system: waiting line + service booth) a time at most equal to his/her patience time (resp. his/her service time plus his/her whole patience time). The set $A$ (resp. $B$) thus contains all the absolute values of the indexes of the customers possibly in the waiting line (resp. in the total system) at time 0. 
% On another hand, the customers spend in the system a time at least equal to their patience time (when they are discarded before service) or their service time (if they are served upon arrival). Consequently, the absolute values of the indices of all the customers present in the system at time 0 all appear in the set $C$. 
Finally, $\tau^+$ counts the number of arrivals customer 0 can see during his/her patience time. %whereas $\nu^+$ counts the minimal number of arrivals customer 0 is able to see during his presence in the system. 

%We provide hereafter several majorizations of $c=\mbox{Card }H$. 
Similarly to Lemma \ref{lemma:majorerho}, %we have that 
\begin{lemma}
\label{lemma:majoretau}
The r.v.'s $\rho$, $\tau^-$ and $\tau^+$ are a.s. finite. In particular, there exist two integer $p$ and $t$ such that
\begin{align}
p&=\min\left\{ n>0 ; \pr{\rho \le n}>0\right\};\label{eq:defp}\\
t&=\min\left\{ n>0 ; \pr{\tau^- \le n}>0\right\}\label{eq:deft-}.
\end{align}
\end{lemma}

\noindent
As $H \subseteq G$ a.s., we have that 
\begin{equation}
\label{eq:battery}
c \le \mbox{Card }\left(L_\alpha \cap \left[Y; Z\right]\right)=\mbox{Card }\left(\N \cap \left[0,\frac{Z-Y}{\alpha}\right]\right)=\left\lceil\frac{Z-Y}{\alpha}\right\rceil\,\mbox{ a.s..}
\end{equation}
Now set a.s.
\begin{align*}
i_0&=\mbox{argmax}\left\{\sigma\circ\theta^{-i}+D\circ\theta^{-i}-\sum_{j=1}^i \xi\circ\theta^{-j}\,;\,i\in \N^*\right\}\\
&=\mbox{argmax}\left\{\sigma\circ\theta^{-i}+D\circ\theta^{-i}-\sum_{j=1}^i \xi\circ\theta^{-j}\,;\,i=1,..,\rho\right\}.
\end{align*}
Then, 
\begin{align*}
Z-Y\le &\,\sigma\circ\theta^{-i_0}+D\circ\theta^{-i_0}-\sum_{j=1}^{i_0} \xi\circ\theta^{-j}\\
&\,\,\,\,\,\,\,\,\,\,\,\,-\left(\left(\sigma\circ\theta^{-i_0}\right)\wedge \left(D\circ\theta^{-i_0}\right)-\sum_{j=1}^{i_0} \xi\circ\theta^{-j}\right)\\
= &\,\,\left(\sigma\circ\theta^{-i_0}\right)\vee \left(D\circ\theta^{-i_0}\right)\,\mbox{ a.s.,}
\end{align*}
so that, with (\ref{eq:battery}),  
\begin{equation*}
% c &\le \left\lfloor\frac{\sup\biggl\{\left(\sigma\circ\theta^{-i}\right)\vee \left(D\circ\theta^{-i}\right)\,;\,i\in\N^*\biggl\}}{\alpha}\right\rfloor +1\nonumber\\
c\le \left\lceil\frac{\max\biggl\{\left(\sigma\circ\theta^{-i}\right)\vee \left(D\circ\theta^{-i}\right)\,;\,i=1,...,\rho\biggl\}}{\alpha}\right\rceil\,\mbox{ a.s.}. 
\end{equation*}
Therefore, on $\{\rho \le p\}$, 
\begin{equation}
\label{eq:majorecard0}
c \le \left\lceil\frac{\underset{i=1,...,p}{\max}\biggl\{\left(\sigma\vee D\right)\circ\theta^{-i}\biggl\}}{\alpha}\right\rceil.
\end{equation}
%and as $c$ is deterministic, the latter upper-bound holds a.s. since it holds with positive probability. 

On another hand, the largest possible workload at time 0 is less than the sum of the service time of the customer in service (whose index has absolute value in 
$B$) and the service times requested by the customers in the waiting line at 0 (their indexes have absolute values in $A$). 
% For any $n$, $H^n$ represents the set of values of the recursion at the origin, 
% starting from a value $x \in \left[Z\circ\theta^{-n},Z\circ\theta^{-n}\right]$, $n$ increments in the past. 
Therefore, a.s., 
\begin{align}
H \subseteq H^n &\subseteq L_\alpha\,\bigcap\,\left(\bigcup_{i \in B} \left[0\,;\,\sigma\circ\theta^{-i}+\sum_{j=1}^{i-1} \left(\sigma\circ\theta^{-j}\right)\car_{A}\{j\}\right]\right)\label{eq:includeH}\\
&\subseteq L_\alpha\,\bigcap\,\left[0\,;\,\sigma\circ\theta^{-\rho}+\sum_{j=1}^{\rho - 1} \left(\sigma\circ\theta^{-j}\right)\car_{A}\{j\}\right]\nonumber\\
&\subseteq L_\alpha\,\bigcap\,\left[0\,;\,\sigma\circ\theta^{-\rho}+\sum_{j=1}^{\tau^-} \left(\sigma\circ\theta^{-j}\right)\car_{A}\{j\}\right]\nonumber\\
&= L_\alpha\,\bigcap\,\left[0;M\right],\nonumber 
\end{align}
where 
\begin{equation}
\label{eq:defM}
M:=\sigma\circ\theta^{-\rho}+\sum_{j=1}^{+\infty} \left(\sigma\circ\theta^{-j}\right)\car_{A}\{j\},
\end{equation}
and where we use the fact that $\tau^- \le \rho,$ a.s.. 
This implies that a.s.
\begin{equation}
\label{eq:cardgeneral}
c  \le 
%\mbox{Card }\left(\N \cap \left[0\,;\,\frac{M}{\alpha}\right]\right)= 
\Bigl\lceil \frac{M}{\alpha}\Bigl\rceil = \frac{M}{\alpha}+1.
\end{equation}
On the event $\{\rho \le p\}$, we have $\tau^- \le p$, thus  
$
M\le \sum_{j=1}^{p} \left(\sigma\circ\theta^{-j}\right),
$
and with (\ref{eq:cardgeneral}), 
\begin{equation}
\label{eq:majorecard1}
c  \le \frac{\displaystyle\sum_{j=1}^{p} \left(\sigma\circ\theta^{-j}\right)}{\alpha}+1.
\end{equation}
Now, on $\{\tau^- \le t\}$, 
$$M \le \sigma\circ\theta^{-\rho}+\sum_{j=1}^{t} \left(\sigma\circ\theta^{-j}\right),$$
so with  (\ref{eq:cardgeneral}),
\begin{equation}
\label{eq:majorecard2}
\displaystyle c \le \frac{\underset{i=1,...,p}{\max}\sigma\circ\theta^{-i}+\sum_{j=1}^{t} \left(\sigma\circ\theta^{-j}\right)}{\alpha}+1.
\end{equation}
The upper bounds (\ref{eq:majorecard0}), (\ref{eq:majorecard1}) and (\ref{eq:majorecard2}) hold with positive probability, hence they are true a.s. since $c$ is deterministic. Therefore, 
\begin{multline*}
c  \le 1+\frac{1}{\alpha}.\min\Biggl\{\underset{i=1,...,p}{\max}\left(\sigma\circ\theta^{-i}\right)+\sum_{j=1}^{t} \left(\sigma\circ\theta^{-j}\right)\\;\sum_{j=1}^{p} \left(\sigma\circ\theta^{-j}\right);\underset{i=1,...,p}{\max}\biggl(\left(\sigma\circ\theta^{-i}\right)\vee \left(D\circ\theta^{-i}\right)\biggl)\Biggl\}\,\mbox{ a.s.}.
\end{multline*}
If we assume in particular that the service times are a.s. bounded, \emph{i.e.} that $\bar s$ defined by (\ref{eq:defS}) is finite, 
we have that
\begin{equation*}
c  \le \bar s\left((t+1)\wedge p\right)+1,
\end{equation*}
and if additionally, the patience times are a.s. bounded (\emph{i.e.} $\bar d$ - defined by (\ref{eq:defD}) - is finite), it follows that
\begin{equation*}
c  \le\bar s\left((t+1)\wedge p\right)\wedge \left\lfloor\bar s \vee \bar d\right\rfloor+1.
\end{equation*}
% 
% Our aim is to upper-bound the cardinal $c$ of $H$ in that case. First, notice that in view of (\ref{eq:inclusionH}), a.s., 
%

\noindent 
Now, remark that for all $j >0$, a.s. 
$$\tau^+\circ\theta^{-j}>j \Longleftrightarrow D\circ\theta^{-j}> \sum_{k=0}^{i-1} \xi\circ\theta^{k-j}  \Longleftrightarrow j \in A.$$
By the very definition of $\tau$,  
$$\sum_{j=0}^{\tau^+-2} \xi\circ\theta^j < D,$$
so taking expectations, and then using $\theta$-invariance we obtain  
\begin{align}
\esp{D} &> \esp{\sum_{j=1}^{\infty} \xi\circ\theta^{j-1}\car_{\tau^+>j}}\nonumber\\
%&= \esp{\sum_{j=1}^{\infty} \frac{\sigma}{\sigma\circ\theta^{j-1}}\xi\circ\theta^{j-1}\car_{\tau^+>j}}\\
&= \esp{\left(\xi\circ\theta^{-1}\right)\sum_{j=1}^{\infty} \car_{\tau^+\circ\theta^{-j}>j}}\nonumber\\
&= \esp{\left(\xi\circ\theta^{-1}\right)\mbox{Card }A}.\label{eq:cargo}
\end{align}
% implies that 
% \begin{align*}
% \esp{D} & > \esp{\frac{\xi\circ\theta^{-1}}{\sigma\circ\theta^{-1}} \sum_{j=1}^{+\infty} \sigma \circ \theta^{-j} \car_{\tau^+ \circ\theta^{-j}> j}}\\
% &=\esp{\frac{\xi\circ\theta^{-1}}{\sigma\circ\theta^{-1}} \sum_{j=1}^{+\infty} \sigma \circ \theta^{-j} \car_A\{j\}}.
% \end{align*}
% But as $\rho \ge \tau^-$, either $\rho=\tau$ and $M$...
Again, if we assume that $\bar s$ is finite, it follows from 
(\ref{eq:defM}) that 
$$M \le \bar s\alpha\left(1+\mbox{Card }A\right)\,\mbox{ a.s.},$$
so with (\ref{eq:cardgeneral}),
$$c \le \bar s\left(1 + \mbox{Card }A\right)+1\,\mbox{ a.s.}.$$
Plugging this into (\ref{eq:cargo}), and using $\theta$-invariance thus yields  
% \begin{equation*}
% \esp{D}>\esp{\xi\circ\theta^{-1}}\left(\frac{\alpha\left(\mbox{Card }H -1\right)}{\bar s\alpha}-1\right),
% \end{equation*}
% or equivalently, since $c$ is an integer,
\begin{equation}
\label{eq:etoile}
c \le \biggl\lceil\frac{\bar s\left(\esp{D}+\esp{\xi}\right)}{\esp{\xi}}\biggl\rceil.
\end{equation}
All these results are collected in the following proposition.
\begin{proposition}
\label{pro:majorec}
A stationary workload exists on the original probability space whenever $\pr{Z = 0}>0$, where $Z$ is defined by (\ref{eq:defZ}). 
If not, if both $\sigma$ and $\xi$ take value in $L_\alpha$ defined by (\ref{eq:defLalpha}), we have that 
\begin{multline}
\label{eq:majorefinal1}
c  \le 1+\frac{1}{\alpha}.\min\Biggl\{\underset{i=1,...,p}{\max}\left(\sigma\circ\theta^{-i}\right)+\sum_{j=1}^{t} \left(\sigma\circ\theta^{-j}\right)\\;\sum_{j=1}^{p} \left(\sigma\circ\theta^{-j}\right);\underset{i=1,...,p}{\max}\biggl(\left(\sigma\vee D\right)\circ\theta^{-i}\biggl)\Biggl\}
\,\mbox{ a.s.},
\end{multline}
where $p$ and $t$ are defined respectively by (\ref{eq:defp}) and (\ref{eq:deft-}). 

If additionnally, $\bar s$ defined by (\ref{eq:defS}) is finite,
\begin{equation*}
%\label{eq:majorefinal2}
c \le \min\Biggl\{\bar s\left((t+1)\wedge p\right)+1\,;\,\biggl\lceil\frac{\bar s\left(\esp{D}+\esp{\xi}\right)}{\esp{\xi}}\biggl\rceil\Biggl\},
\end{equation*}
and if $\bar d$ defined by (\ref{eq:defD}) is as well finite,
\begin{equation*}
%\label{eq:majorefinal3}  
c \le \left\lfloor\bar s \vee \bar d\right\rfloor+1.
\end{equation*}  
\end{proposition}

Let us go through two examples to illustrate how our extension technique can be used to solve the stability problem of the queue with impatient 
customers. We work on the same dynamical system as in Examples 1 and 2.
\textit{ }\\\\
 
\noindent\textbf{Example 3}\\
\noindent
%\label{example:1}
We first consider (a particular case of) Example 1, p. 303 in \cite{Moy09}. We set on $\left(\Omega,\maF,\mathbf P, \theta\right)$ the r.v.:
\[\left\{\begin{array}{ll}
\xi (\om_1) &=\xi(\om_2)=1;\\
\sigma(\om_1) &=0.5,\,\sigma (\om_2)=1.5;\\
D(\om_1) &=1.51,\,D(\om_2)=2.01.
\end{array}\right.\]
The workload is then valued in $0.5 \N$, and 
we check that for this model, according to the definitions (\ref{eq:defY}) and (\ref{eq:defZ}), 
\begin{align*} 
Z(\om_1)&=\max\biggl\{\sigma(\om_2)+D(\om_2)-\xi(\om_2)\,;\,\sigma(\om_1)+D(\om_1)-\left(\xi(\om_2)+\xi(\om_1)\right)\,;\,
...\biggl\}\\
&=2.51;\\
Z(\om_2)&=\max\biggl\{\sigma(\om_1)+D(\om_1)-\xi(\om_1)\,;\,\sigma(\om_2)+D(\om_2)-\left(\xi(\om_1)+\xi(\om_2)\right)\,;\,
...\biggl\}\\
&=1.51;\\
Y(\om_1) &=0.5;\\
Y(\om_2) &=0.
\end{align*}

We thus start from the set 
$G=[Y,Z] \cap 0.5 \N$, so 
$$G_{\om_1}=\left\{0.5,1,...,2.5\right\},\,G_{\om_2}=\left\{0,0.5,...,1.5\right\}.$$
It can be checked as well that 
\begin{align*}
A_{\om_1}&=\{1\}\,\mbox{ and }A_{\om_2}=\{1,2\};\\
B_{\om_1}&=\{1,2,3\}\,\mbox{ and }B_{\om_2}=\{1,2\},
\end{align*}
so $t=1$ and $p=2$. 
The least upper-bound of $c$ obtained according to Proposition \ref{pro:majorec} is given by 
$$\lfloor \bar s \vee \bar d\rfloor+1=\lfloor 1.5 \vee 2.01\rfloor+1=3.$$
To explicit construct the extension, we now form the random sequence $\{H^n\}$ for the sample $\om_1$:
\begin{align*}
H^1_{\om_1}&=\varphi_{\om_2}\left(G_{\om_2}\right)=\left\{0.5,1,1.5,2\right\};\\
H^2_{\om_1}&=\varphi_{\om_2}\circ\varphi_{\om_1}\left(G_{\om_1}\right)=\varphi_{\om_2}\left(\left\{0,0.5,1,1.5\right\}\right)=\left\{0.5,1,1.5,2\right\};\\
H^3_{\om_1}&=\varphi_{\om_2}\circ\varphi_{\om_1}\circ\varphi_{\om_2}\left(G_{\om_2}\right)=\varphi_{\om_2}\circ\varphi_{\om_1}\left(\left\{0.5,
1,1.5,2\right\}\right)=\left\{0.5,1,1.5\right\};\\
H^4_{\om_1}&=\varphi_{\om_2}\circ\varphi_{\om_1}\circ\varphi_{\om_2}\circ\varphi_{\om_1}\left(G_{\om_1}\right)=\varphi_{\om_2}\circ\varphi_{\om_1}
\left(\left\{0.5,1,1.5,2\right\}\right)=\left\{0.5,1,1.5\right\};\\
H^5_{\om_1}&=\varphi_{\om_2}\circ\varphi_{\om_1}\left(\left\{0.5,1,1.5\right\}\right)=\left\{0.5,1,1.5\right\};\\
H^6_{\om_1}&=\varphi_{\om_2}\circ\varphi_{\om_1}\left(\left\{0.5,1,1.5\right\}\right)=\left\{0.5,1,1.5\right\};\\
\vdots &
\end{align*}
Therefore, $H_{\om_1}=\left\{0.5,1,1.5\right\}$, and we obtain likewise that 
$H_{\om_2}=\left\{0,0.5,1\right\}$. So $c=3$ and there are at most three solutions on the original space. 

In fact, it can be checked that the following events of $\tilde{\maF}$: 
\begin{align*}
\mathcal I &=\left\{(\om_1,0.5)\,;\,(\om_2,0)\right\},\\
\mathcal I' &=\left\{(\om_1,1)\,;\,(\om_2,0.5)\right\},\\
\mathcal I'' &=\left\{(\om_1,1.5)\,;\,(\om_2,1)\right\}\\
\end{align*}
all belong to the set $\mathcal K$ of Lemma \ref{lemma:isomorphe}. As a consequence, the extension is not ergodic, and the three corresponding 
 r.v.'s $X$, $X'$ and $X''$ are the only three stationary workloads on the original space.
\bigskip

\noindent\textbf{Example 4}\\
\noindent Now, set on the same probability space: 
\[\left\{\begin{array}{ll}
\xi (\om_1) &=\xi(\om_2)=1;\\
\sigma(\om_1) &=3,\,\sigma (\om_2)=2;\\
D(\om_1) &=3.01,\,D(\om_2)=1.99.
\end{array}\right.\]
The workload sequence is then valued in $\N$. As above, we first check that 
\begin{align*}
Z(\om_1)=4.01&\mbox{ and } Z(\om_2)=5.01,\\
Y(\om_1)=1&\mbox{ and } Y(\om_2)=2,\\
A_{\om_1}=\{1,2\}&\mbox{ and } A_{\om_2}=\{1,3\},\\
B_{\om_1}=\{1,2,3,4,6\}&\mbox{ and } B_{\om_2}=\{1,2,3,5\},
\end{align*}
so that $t=2$ and $p=5$. So the smallest upper-bound of $c$ given by Proposition \ref{pro:majorec} is 
$$\lfloor\bar s \vee \bar d\rfloor+1=4.$$

Once again, we set $G=[Y,Z]\cap \N\,\mbox{ a.s.,}$ 
which amounts to
$$G_{\om_1}=\left\{1,2,3,4\right\}\mbox{ and }G_{\om_2}=\left\{2,3,4,5\right\}.$$
Then the computation yields  
$$H_{\om_1}=\left\{2,3,4\right\}\mbox{ and }H_{\om_2}=\left\{3,4,5\right\},$$
therefore $c=3$. It can be checked in that case, that the invariant sigma field of $\tilde{\maF}$ is 
$\left\{\emptyset,\tilde\Omega\right\}$, so the extension 
is ergodic, but there is no stationary workload on the original probability space. 
%there is no couple $(i,j) \in H_{\om_1}\times H_{\om_2}$ such that $\varphi_{\om_1}(i)=j$ and 
%$\varphi_{\om_2}(j)=i$: the set $L$ is empty for this recursion, so there is no solution on the original space. 

\subsection*{Independent case}
%\label{subsubsec:indep}
We now address the case where the service times, patience times and inter-arrivals times form three independent i.i.d. sequences: the system is then  denoted GI/GI/1/1+GI, and has been thoroughly studied \emph{e.g.} by Baccelli \emph{et al.} (\cite{MR86d:60103,HebBac81}). 

First, assume that 
\begin{equation}
\label{eq:condsuffGIGI}
\pr{\sigma < \xi} > 0,
\end{equation}
which implies, by independence, that $\pr{\sigma \le y-\varepsilon}>0$ and $\pr{\xi \ge y}>0$ for some $y, \varepsilon >0$. 

The r.v. $Z$ defined by (\ref{eq:defZ}) is finite, so there exists $n \in \N$ such that 
$\pr{Z < n\varepsilon}>0$. 
In particular, by $\theta$-invariance, 
$$\pr{Z\circ\theta^{-n} < n\varepsilon}>0.$$
Denote now for all $i \in \N^*$ and all $x\in\R+$ the events
\begin{align*}
\mathcal E_i^x &=\left\{\xi\circ \theta^{-j} \ge x;\,\forall j=1,...,i\right\};\\
\mathcal F_i^x &=\left\{\sigma\circ \theta^{-j}\le x;\,\forall j=1,...,i\right\},\\
%G_n^x &=\left\{\sigma\circ \theta^{-p}+D\circ\theta^{-p} \le x;\,\forall p=1,...,n\right\};\\
\end{align*}
fix a sample on the event
$$\mathcal A_n:=\left\{Z\circ \theta^{-n} < n\varepsilon\right\} \cap \mathcal E_n^y \cap \mathcal F_n^{y-\varepsilon},$$
and assume that customer $-n$ finds upon arrival a workload $w$ such that 
$$w \in G_{\theta^{-n}\om}=L_{\alpha} \cap \left[Y\left(\theta^{-n}\om\right),\,Z\left(\theta^{-n}\om\right)\right].$$  
%In particular, $w \le Z\left(\theta^{-n}\om\right) < n\varepsilon$, hence 
We are in the following alternative:
\begin{itemize}
\item[(i)] either for all $q=1,...,n-1$, the workload $\Phi^{n-q}_{\theta^{-q}\om}(w)$ upon the arrival of $C_{-q}$ is positive, so 
the server never idles before the end of service of customer $C_{-1}$. In that case, the workload $\Phi^n_\om(w)$ at 0 is less than 
the workload upon the arrival of customer $-n$ plus the work brought by the customers $C_{-n},C_{-(n-1)},...,C_{-1}$ minus the time elapsed, in other words
$$\Phi^n_{\om}(w) \le \left[w + \sum_{j=1}^n \sigma\circ\theta^{-j}(\om)-\sum_{j=1}^n \xi\circ\theta^{-j}(\om)\right]^+.$$
Hence, since $\om \in \mathcal A_n$, 
\begin{align*}
\Phi^n_{\om}(w)&\le \left[Z\left(\theta^{-n}\om\right) + n(y-\varepsilon) - ny\right]^+\\
& =0. 
\end{align*}
\item[(ii)] or for some $q\in \left\{1,...,n-1\right\}$ (take the largest one), $\Phi^{n-q}_{\theta^{-q}\om}(w)=0$, the system is empty at the arrival of $C_{-q}$. Hence, since $$\sigma\circ\theta^{-j}(\om) < \xi\circ\theta^{-j}(\om);\,j=1,...,q,$$ each following customer is then immediately 
attended upon arrival and leaves the system before the next arrival, so $\Phi^{n-j}_{\theta^{-j}\om}(w)=0$ for all $j \in \{0,...,q\}$, and in particular $\Phi^{n}_{\om}(w)=0$.
\end{itemize}
Therefore, in any case and for any $w \in G_{\theta^{-n}\om}$, $\Phi^{n}_{\om}(w)=0$. Thus, $H=\{0\}$, and in particular $c=1$,
 on $\mathcal A_n$. 
It is now easy to check that $\pr{\mathcal A_n}>0$ since the events $\left\{Z\circ \theta^{-n} < n\varepsilon\right\}$, $\mathcal E_n^y$ and $\mathcal F_n^{y-\varepsilon}$ are clearly independent and of positive probability. Thus, $c=1$ a.s., the extension is ergodic and there exists a unique solution $X$ on the original probability space, for which $\pr{X=0}>0$. We hence capture again  
the stability result of Baccelli \emph{et al.} (see \cite{MR86d:60103,HebBac81}). 

Assume now that (\ref{eq:condsuffGIGI}) does not hold.
Notice that $$\pr{G_n^{\underaccent{\bar}{s}\alpha}}=\left(\pr{\sigma \le \underaccent{\bar}{s}\alpha}\right)^n>0,$$
where $\underaccent{\bar}{s}$ is defined by (\ref{eq:defs}).  
As above, it is then easily checked that the events $\{\rho \le n\}$ and $G_n^{\underaccent{\bar}{s}\alpha}$ are independent, thus 
$$\pr{\left\{\rho \le n\right\} \cap G_n^{\underaccent{\bar}{s}\alpha}} > 0.$$
On the latter event, $M$ defined by (\ref{eq:defM}) is such that 
$M \le x\left(1+\mbox{Card }A\right)$. Therefore the argument leading to (\ref{eq:etoile})yields  
$$\mbox{Card }H \le \biggl\lceil\frac{\underaccent{\bar}{s}\left(\esp{D}+\esp{\xi}\right)}{\esp{\xi}}\biggl\rceil.$$

%\bibliographystyle{amsalpha}
%\bibliography{bibliographie4}
\providecommand{\bysame}{\leavevmode\hbox to3em{\hrulefill}\thinspace}
\providecommand{\MR}{\relax\ifhmode\unskip\space\fi MR }
\providecommand{\MRhref}[2]{%
  \href{http://www.ams.org/mathscinet-getitem?mr=#1}{#2}
}
\providecommand{\href}[2]{#2}

\end{document}